\PassOptionsToPackage{svgnames,dvipsnames}{xcolor}

\documentclass[EJP,preprint]{ejpecp} %

\usepackage[T1]{fontenc}
\usepackage[utf8]{inputenc}

\usepackage{bbm}
\usepackage{commath}
\usepackage{enumitem}
\usepackage{xspace}
\usepackage[config, labelfont={normalsize}]{caption,subfig}
\usepackage{tikz}

\SHORTTITLE{Modulus of continuity of Kerov transition measure}

\TITLE{Modulus of continuity of  
		Kerov transition measure \\ for continual Young diagrams%
\support{Research 
			was supported by Narodowe Centrum Nauki, grant number 2017/26/A/ST1/00189.}%
\thanks{We thank Maciej Dołęga, Łukasz Maślanka and Dan Romik for discussions and
			suggestions concerning the bibliography.}}

\AUTHORS{%
  Piotr Śniady\footnote{Institute of Mathematics,
  	Polish Academy of Sciences,
  	{\'S}niadeckich~8, 00-656 Warszawa, Poland \\
    \EMAIL{psniady@impan.pl}}\orcid{0000-0002-1356-4820}
}


\AMSSUBJ{20C30} %
\AMSSUBJSECONDARY{%
44A60;  %
60E10;   %
05E10;  %
20C32%
} %

\SUBMITTED{May 29, 2024} %
\ACCEPTED{June 9, 2025} %

\ARXIVID{2405.07091} %

\VOLUME{0}
\YEAR{2023}
\PAPERNUM{0}
\DOI{10.1214/YY-TN}

\ABSTRACT{The \emph{transition measure} is a foundational concept introduced by Sergey
Kerov to represent the shape of a Young diagram as a centered probability
measure on the real line. Over a period of decades the transition measure
turned out to be an invaluable tool for many problems of the asymptotic
representation theory of the symmetric groups. Kerov also showed how to expand
this notion for a wider class of \emph{continual} diagrams so that the
transition measure provides a homeomorphism between a subclass of continual
diagrams (having a specific support) and a class of centered probability
measures with a support contained in a specific interval. We quantify the
modulus of continuity of this homeomorphism. More specifically, we study the
dependence of the cumulative distribution function of Kerov transition measure
on the profile of a diagram at the locations where the profile is not too
steep.}

\newcommand{\R}{\mathbb{R}}
\newcommand{\Z}{\mathbb{Z}}
\newcommand{\cauchy}{\mathbf{G}}

\newcommand{\kerx}{\mathbbm{x}}
\newcommand{\kery}{\mathbbm{y}}
\newcommand{\kerell}{\mathbb{L}}

\newcommand{\conti}{continual\xspace}
\newcommand{\Conti}{Continual\xspace}

\newcommand{\dxy}{d_{XY}}

\newcommand{\smallomega}{\omega}
\newcommand{\referenceomega}{\overline{\Omega}}
\newcommand{\bigomega}{\Omega}
\newcommand{\zzeroplus}{z_+^{\max}}
\newcommand{\zzerominus}{z_-^{\min}}

\DeclareMathOperator{\AS}{AS}

\makeatletter
\newcommand{\@giventhatstar}[2]{\left[#1\;\middle|\;#2\right]}
\newcommand{\@giventhatnostar}[3][]{#1(#2\;#1|\;#3#1)}
\newcommand{\giventhat}{\@ifstar\@giventhatstar\@giventhatnostar}
\makeatother

\begin{document}

\section{Introduction}
\subsection{Young diagrams}

\subsubsection{The French convention}

A Young diagram is a finite collection of boxes arranged in the positive
quarter-plane, aligned to the left and bottom (see Figure \ref{subfig:french}).
This particular way of drawing Young diagrams is known as \emph{the French
	convention}. For a Young diagram with $\ell$ rows, we associate an integer
partition $\lambda = (\lambda_1, \dots, \lambda_\ell)$, where $\lambda_j$
represents the number of boxes in the $j$-th row (counted from bottom to top).
Essentially, we identify a Young diagram with its corresponding partition
$\lambda$.

The motivation for studying Young diagrams lies in their natural occurrence in
representation theory. Specifically, they play a crucial role in the
representation theory of symmetric groups as well as in the representation
theory of general linear groups \cite{FultonHarris}.

\subsubsection{The Russian convention}

For asymptotic problems, it is convenient to draw the Young diagrams in the
\emph{Russian convention}, as shown in Figure \ref{subfig:russian}. This
convention corresponds to the coordinate system $(u,v)$ that is related to
the usual (French) Cartesian coordinates by
\begin{equation}
	\label{eq:russian}
	u = x - y, \qquad\qquad v = x + y. 
\end{equation}

The boundary of a Young diagram $\lambda$ is called its \emph{profile}, as
depicted in Figure \ref{subfig:french}. In the Russian coordinate system, the
profile can be seen as the plot of the function $\omega_\lambda \colon
\mathbb{R} \to \mathbb{R}_+$, as shown in Figure \ref{subfig:russian}.

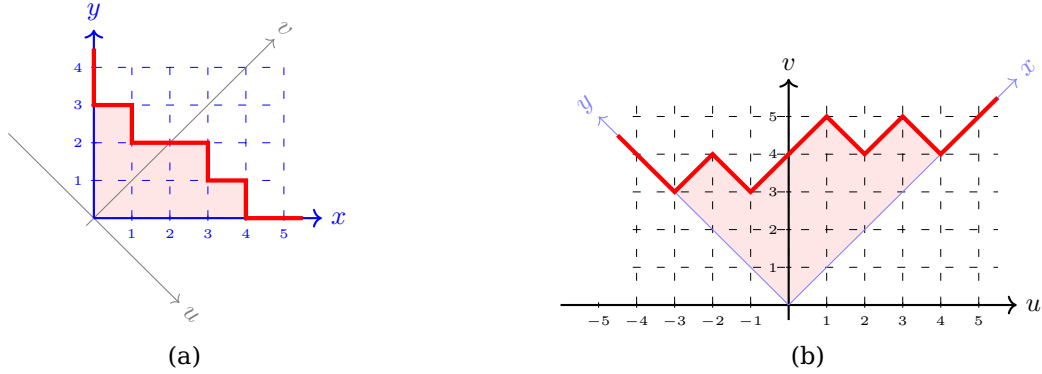
\begin{figure}
	\subfloat[]{
	\begin{tikzpicture}
		
		\begin{scope}[scale=0.5/sqrt(2),rotate=-45,draw=gray]
			
			\begin{scope}[draw=gray,rotate=45,scale=sqrt(2)]
				\fill[fill=red!10] (4,0) -- (4,1) -- (3,1) -- (3,2) -- (1,2)
				-- (1,3) -- (0,3) -- (0,0) -- cycle ;
			\end{scope}
			
			\begin{scope}[rotate=45,draw=blue,scale=sqrt(2)]
				\draw[ultra thin, loosely dashed] (0,0) grid (5,4);
			\end{scope}
			
			\draw[->,thin] (-4.5,0) -- (4.5,0)
			node[anchor=west,rotate=-45]{\textcolor{gray}{$u$}};
			
			\draw[->,thin] (0,-0.4) -- (0,9.5)
			node[anchor=south,rotate=-45]{\textcolor{gray}{$v$}};
			
			\begin{scope}[draw=blue,rotate=45,scale=sqrt(2)]
				
				\draw[->,thick] (0,0) -- (6,0) node[anchor=west]{\textcolor{blue}{$x$}};
				\foreach \x in {1, 2, 3, 4, 5}
				{ \draw (\x, -2pt) node[anchor=north] {\textcolor{blue}{\tiny{$\x$}}} -- (\x,
					2pt); }
				
				\draw[->,thick] (0,0) -- (0,5) node[anchor=south] {\textcolor{blue}{$y$}};
				\foreach \y in {1, 2, 3, 4}
				{ \draw (-2pt,\y) node[anchor=east] {\textcolor{blue}{\tiny{$\y$}}} -- (2pt,\y); }
				
				\draw[ultra thick,draw=red] (5.5,0) -- (4,0) -- (4,1) -- (3,1) --
				(3,2) -- (1,2) -- (1,3) -- (0,3) -- (0,4.5) ;
				
			\end{scope}
			
		\end{scope}
	\end{tikzpicture}
	\label{subfig:french}
}
\hfill
\subfloat[]
{
	\begin{tikzpicture}
		\begin{scope}[xshift=7cm, yshift=-0.5cm, scale=0.5]
			
			\begin{scope}[draw=gray,rotate=45,scale=sqrt(2)]
				\fill[fill=red!10] (4,0) -- (4,1) -- (3,1) -- (3,2) -- (1,2)
				-- (1,3) -- (0,3) -- (0,0) -- cycle ;
			\end{scope}
			
			\begin{scope}
				\clip (-4.5,0) rectangle (5.5,5.5);
				\draw[ultra thin, loosely dashed] (-6,0.01) grid (6,6);
			\end{scope}
			
			\draw[->,thick] (-6,0) -- (6,0) node[anchor=west]{$u$};
			\foreach \z in {-5, -4, -3, -2, -1, 1, 2, 3, 4, 5}
			{ \draw (\z, -2pt) node[anchor=north] {\tiny{$\z$}} -- (\z, 2pt); }
			
			\draw[->,thick] (0,-0.4) -- (0,6) node[anchor=south]{$v$};
			\foreach \t in {1, 2, 3, 4, 5}
			{ \draw (-2pt,\t) node[anchor=east] {\tiny{$\t$}} -- (2pt,\t); }

			\begin{scope}[draw=blue!50,rotate=45,scale=sqrt(2)]
				
				\draw[->,thin] (0,0) -- (6,0) node[anchor=west,rotate=45]
				{\textcolor{blue!50}{{$x$}}};
				
				\draw[->,thin] (0,0) -- (0,5) node[anchor=south,rotate=45]
				{\textcolor{blue!50}{{$y$}}};
				
				\draw[ultra thick,draw=red] (5.5,0) -- (4,0) -- (4,1) -- (3,1) --
				(3,2) -- (1,2) -- (1,3) -- (0,3) -- (0,4.5) ;
				
			\end{scope}
		\end{scope}
		
	\end{tikzpicture}
	\label{subfig:russian}
}
	\caption
	{
		The Young diagram $\lambda=(4,3,1)$ shown in
		\protect\subref{subfig:french}~the French convention, and 
		\protect\subref{subfig:russian}~the Russian convention. The thick solid
		red zigzag line
		represents the \emph{profile} $\omega_\lambda$ of the Young diagram. The
		coordinates system
		$(u,v)$ corresponding to the Russian convention and the coordinate
		system
		$(x,y)$ corresponding to the French convention are shown.
	}
	\label{fig:french} 	
\end{figure}

\subsection{Transition measure of a Young diagram}
\label{sec:transition-measure}

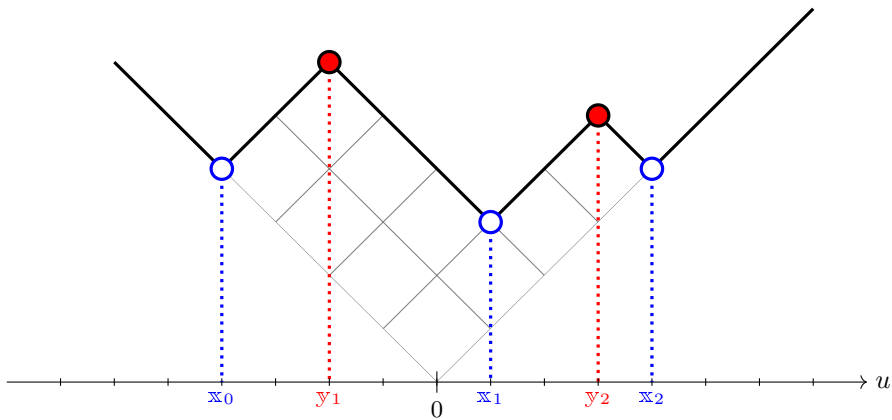
\begin{figure}
		\begin{tikzpicture}[scale=1,rotate=45]
		\coordinate (p0) at (4,0);
		\coordinate (p1) at (4,1);
		\coordinate (p2) at (2,1);
		\coordinate (p3) at (2,4);
		\coordinate (p4) at (0,4);
		\coordinate (p5) at (0,6);
		\begin{scope}
			\clip  (p0) -- (p1) -- (p2) -- (p3) -- (p4) -- (0,0);
			\draw[black!50] (0,0) grid (10,10);
		\end{scope}
		\draw[very thick] (7,0) -- (p0) -- (p1) -- (p2) -- (p3) -- (p4) -- (p5);
		\draw[->] (-4,4) -- (4,-4) node[anchor=west] {$u$};
		\draw[dotted,blue,very thick] (p0) -- (4/2,-4/2) node[anchor=north] {\small $\kerx_2$};
		\draw[dotted,blue,very thick] (p2) -- (1/2,-1/2) node[anchor=north] {\small $\kerx_1$};
		\draw[dotted,blue,very thick] (p4) -- (-4/2,4/2) node[anchor=north] {\small $\kerx_0$};
		\draw[dotted,red,very thick] (p1) -- (3/2,-3/2) node[anchor=north] {\small $\kery_2$};
		\draw[dotted,red,very thick] (p3) -- (-2/2,2/2) node[anchor=north] {\small $\kery_1$};
		\draw (0,0) +(0.1,0.1) -- +(-0.1,-0.1) node[anchor=north] {\small $0$};
		\draw[blue,very thick,fill=white] 
		(p0) circle (4pt)
		(p2) circle (4pt)
		(p4) circle (4pt);
		\draw[very thick,fill=red] 
		(p1) circle (4pt)
		(p3) circle (4pt);
		
		\foreach \u in {-7,...,7}
		{\draw (\u/2,-\u/2)  +(-1pt,-1pt) -- +(1pt,1pt); }
	\end{tikzpicture}
	\caption{Concave corners (blue and empty) and convex corners (red and
		filled) 
		of the Young diagram $(4,2,2,2)$ and their $u$-coordinates.}
	\label{fig:transition}
\end{figure}

For a given Young diagram~$\lambda$ we denote by
$\kerx_0<\cdots<\kerx_{\kerell}$ the
$u$-coordinates of its \emph{concave corners} and by
$\kery_1<\cdots<\kery_\kerell$ the
$u$-coordinates of its \emph{convex corners}, see Figure~\ref{fig:transition}. The
\emph{Cauchy transform} of $\lambda$ is defined as the rational
function 
\begin{equation}
	\label{eq:cauchy-def}     
	\cauchy_\lambda(z)=
	\frac{(z-\kery_1)\cdots(z-\kery_\kerell)}{(z-\kerx_0)\cdots(z-\kerx_\kerell)},
\end{equation}
see \cite{Kerov1993} and \cite[Chapter~4, Section~1]{KerovBook}. Note that in
the work of Kerov this function is called \emph{the generating function} of
$\lambda$. The Cauchy transform can be uniquely written as a sum of simple
fractions:
$$ \cauchy_\lambda(z)= \sum_{0\leq i \leq \kerell}  \frac{p_i}{z-\kerx_i} $$
with the positive coefficients $p_0,\dots,p_\kerell>0$ such that
$p_0+\cdots+p_\kerell=1$.
The \emph{transition measure} of $\lambda$ is defined 
as the discrete measure
\begin{equation}
	\label{eq:transition-def}
	\mu_\lambda= p_0\ \delta_{\kerx_0} + \cdots + p_\kerell\ \delta_{\kerx_\kerell}
\end{equation} 
which to the atom $\kerx_i$ associates the probability $p_i$.
In this way the rational function defined by \eqref{eq:cauchy-def} fulfills
$$ \cauchy_\lambda(z) = \int_{\R} \frac{1}{z-x} \dif \mu_\lambda(x),$$
so it is indeed the Cauchy transform of the measure  $\mu_\lambda$.

\subsection{Motivations and context for transition measure}

The definition of the transition measure presented in
Section~\ref{sec:transition-measure} has an advantage of being self-contained, but it
does not provide much insight into the motivations behind this concept. We will
recall some equivalent definitions of the transition measure which show its
relevance for probabilistic and algebraic methods of the representation theory
of the symmetric groups. We will also recall briefly some of its applications.

\subsubsection{Plancherel growth process}

The \emph{Plancherel growth process} is a natural Markov chain
$(\lambda^{(n)})$ on the set of Young diagrams. This random walk starts with
$\lambda^{(0)}=\emptyset$ being the empty diagram. At each step, the random
Young diagram $\lambda^{(n+1)}$ is obtained by adding a single box to one of
the concave corners of the previous Young diagram $\lambda^{(n)}$. If $\mu$ is
a Young diagram obtained from $\lambda$ by adding a single box, the transition
probability for this Markov chain
\begin{equation} 
	\label{eq:growth}
	\mathbb{P}\giventhat*{ \lambda^{(n+1)} = \mu }{ \lambda^{(n)}=\lambda } 
	= \frac{f^{\mu} }{(n+1) f^{\lambda}} 
\end{equation} 
is defined in a natural way in terms of the dimensions $f^\lambda$ and $f^\mu$
of \emph{the irreducible representations of the symmetric groups} defined by,
respectively, $\lambda$ and $\mu$. 

Equivalently, the Plancherel growth process
can be described as the outcome of the Robinson--Schensted--Knuth algorithm
(RSK) applied to a sequence of independent, identically distributed random
variables with the uniform distribution on the unit interval~$[0,1]$. 

For a
pedagogical introduction to this topic we refer to the book of Romik
\cite[Section~1.19]{Romik2015}.

\subsubsection{Plancherel growth process and the transition measure}
\label{sec:ktm-and-plancherel}

Let us fix a Young diagram $\lambda$; we will keep the notations from
Figure~\ref{fig:transition}. The collection of the transition probabilities
\eqref{eq:growth} over all valid choices of $\mu$ can be encoded by a discrete
probability measure on the real line, given as follows.  To a real number
$\kerx_i$ which corresponds to some concave corner of $\lambda$ we associate the
transition probability \eqref{eq:growth} from $\lambda$ to the diagram $\mu$
obtained from $\lambda$ by adding a single box at the corner $\kerx_i$.

Thanks to the hook-length formula for the dimensions of the irreducible
representations, the atoms of this discrete measure (each atom is given by the
right-hand side of \eqref{eq:growth}) can be explicitly calculated. After a heavy
cancellation of the factors in the hook-length formula for $f^\lambda$ and
$f^\mu$ this transition probability turns out to coincide with the
corresponding atom $p_i$ prescribed in Section~\ref{sec:transition-measure}.

\medskip

This was the original motivation for introducing the transition measure. Indeed,
transition measure is a convenient tool for the probabilistic and ergodic
approach to the problems of the (asymptotic) representation theory of the
symmetric groups.

\subsubsection{Transition measure and Jucys--Murphy elements} 

The \emph{Jucys--Murphy element}
\[ J_{n+1} = (1, n+1) + \cdots + (n, n+1) \in \mathbb{C}[S_{n+1}] \]
is an element of the group algebra of the symmetric group $S_{n+1}$,
given by the formal sum of the transpositions interchanging $n+1$ with the
smaller natural numbers. Jucys--Murphy elements are fundamental for the modern
approach to the representation theory of the symmetric groups, see 
\cite{Ceccherini-SilbersteinScarabotti2010} for a pedagogical introduction.

It was noticed by Biane \cite[Section~3.3]{Biane1998} that the transition
measure of a Young diagram $\lambda$ with $n$ boxes can be equivalently defined
as the spectral measure of the Jucys--Murphy element $J_{n+1}$ with respect to
some specific state defined in terms of the irreducible character of $\lambda$.
The result
of Biane implies that each moment 
\[ \int x^k \dif \mu_\lambda(x) \]
of the transition
measure is directly related to the character of $\lambda$ evaluated on (some
specific conditional expectation applied to) the power
$J_{n+1}^k$ of the Jucys--Murphy element.

\subsubsection{Transition measure and the characters of symmetric groups} This
link opened the possibility to look for exact formulas which would express the
irreducible characters of the symmetric group in terms of the transition measure
of the corresponding Young diagram. The advantage of such formulas (over some
classical tools such as \emph{the Murnaghan--Nakayama rule}) would lie in their
low computational complexity which would make them perfect for some asymptotic
problems of the representation theory.

It was soon noticed that the best quantities for this endeavor are not the
moments of the transition measure, but \emph{free cumulants}, slightly more
complex quantities which have their origin in the random matrix theory and
Voiculescu's non-commutative free probability \cite{NicaSpeicher-lectures}. The
irreducible characters can be expressed as polynomials, called \emph{Kerov
	polynomials}, in terms of the free cumulants of the corresponding transition
measure \cite{DFS}.

This explains why transition measure and its free cumulants turned out to be
important tools for probabilistic problems of the asymptotic representation
theory of the symmetric groups \cite{Biane1998,Biane2001,Sniady2006a}.

\subsection{\Conti diagrams}

\subsubsection{Motivations: limit theorems} One of the branches of the
asymptotic representation theory concerns limit theorems for large random Young
diagrams. For example, for each integer $n\geq 1$ let $\lambda^{(n)}$ be a
random Young diagram with $n$ boxes, with some specified probability
distribution. Informally speaking, we would like to consider a \emph{`resized
	Young diagram'}
\begin{equation}
	\label{eq:resized}
	\frac{1}{\sqrt{n}} \lambda^{(n)} 
\end{equation} 
which is obtained by drawing the boxes of $\lambda^{(n)}$ as squares with the
size $\frac{1}{\sqrt{n}}$, so that the total area of \eqref{eq:resized} is
equal to $1$. Unfortunately, this object does not make sense as a Young
diagram.

Even though the \emph{`resized Young diagram'}
\eqref{eq:resized} does not exist, it makes sense to speak about its
\emph{profile} which is defined as the homothety of the profile of the original
diagram~$\lambda^{(n)}$:
\begin{equation}
	\label{eq:resized-profile}
	\omega_{\left( \frac{1}{\sqrt{n}} \lambda^{(n)} \right)}
	(z) 
	= \frac{1}{\sqrt{n}}\  \omega_{
		\lambda^{(n)}}\!\left( \sqrt{n}\ z \right).   
\end{equation} 
A typical problem in the asymptotic representation theory would be to ask if the
sequence of such resized profiles \eqref{eq:resized-profile} converges to some
limit \cite{LoganShepp1977,VershikKerov1977}. This was the original motivation
for the notion of \emph{\conti diagrams} which encompasses the resized
profiles \eqref{eq:resized-profile} as well as their limits. We present the
details in the following.

\subsubsection{\Conti diagrams}
\label{sec:continuous}

We say that $\omega\colon \R \to \R_+$ is a \emph{\conti diagram} if the
following two conditions are satisfied:
\begin{itemize}
	\item $|\omega(z_1) - \omega(z_2) | \leq | z_1 - z_2|$ holds true for any
	$z_1,z_2\in\R$,
	\item $\omega(z) = |z|$ for sufficiently large $|z|$.
\end{itemize}
In the literature, such \conti diagrams are also referred to as
\emph{generalized diagrams}, \emph{continuous diagrams}, or simply
\emph{diagrams}. Notably, an important class of examples arises from the
profiles $\omega_{\lambda}$ associated with the usual Young diagrams.

\begin{remark}
	In his original paper \cite{Kerov1993}, Kerov considered a more general
	form for the second condition. Specifically, he required that \mbox{$\omega(z) =
		|z - z_0|$} should hold for some fixed value of $z_0\in\R$ and for
	sufficiently
	large $|z|$. The definition we use corresponds to the specific choice of
	$z_0 = 0$. Consequently, the \emph{\conti diagrams} discussed in the
	current paper
	align with what Kerov terms \emph{centered diagrams}.
\end{remark}

\subsection{Transition measure for \conti diagrams}

\subsubsection{Transition measure for zigzag  diagrams}
\label{sec:zigzag}

We will say that a \conti diagram $\omega$ is a \emph{zigzag} if
$\omega\colon\R\to\R_+$ is a piece-wise affine function with the slopes being
only $\pm 1$. Heuristically, this means that $\omega$ is a zigzag line as
depicted on Figure~\ref{fig:transition}. The only difference is that now the
$u$-coordinates of its concave corners $\kerx_0<\cdots<\kerx_{\kerell}$ as well
as the $u$-coordinates of its convex corners $\kery_1<\cdots<\kery_\kerell$ are
no longer assumed to be integer numbers.

For a zigzag diagram $\omega$ we define its Cauchy
transform $\cauchy_\omega$ in the same way as before, by the formula
\eqref{eq:cauchy-def}. The transition measure of~$\omega$, denoted by
$\mu_\omega$ is defined in the same way as before, by the formula
\eqref{eq:transition-def}.

\subsubsection{Transition measure for generic \conti diagrams} 

Let $\omega$ be a zigzag, we keep the notations from above. The second
derivative $\omega''$ is well-defined as a Schwartz distribution and can be
identified with a signed measure on the real line. The positive part of this
measure is supported in the concave corners $\kerx_0,\dots,\kerx_{\kerell}$
while the negative part of this measure is supported in the convex corners
$\kery_1,\dots,\kery_\kerell$, with each atom having equal weight~$2$. An
application of logarithm transforms the product on the right-hand side of
\eqref{eq:cauchy-def} into a sum. It was observed by Kerov that it can be
conveniently written in the form
\begin{multline}
	\label{eq:transition-general}
	\log\big[ z \cauchy_\omega(z) \big]= - 
	\int_{-\infty}^\infty \log (z-w) \ \left(\frac{\omega(w)-|w|}{2}\right)'' \dif w  \\
	= - \int_{-\infty}^\infty \frac{1}{z-w} \ \left(\frac{\omega(w)-|w|}{2}\right)' \dif w,
\end{multline}
where $z$ is a complex variable.

We can drop the assumption that $\omega$ is a zigzag and use
\eqref{eq:transition-general} to define the Cauchy transform $\cauchy_\omega$
for an arbitrary \conti diagram $\omega$. The transition measure~$\mu_\omega$ is
then defined from the Cauchy transform $\cauchy_\omega$ via Stieltjes inversion
formula. The \emph{cumulative function} $K_\omega$ of a \conti diagram $\omega$
is defined as the cumulative distribution function of the corresponding
transition measure
\[  K_\omega(z)= \mu_\omega\big( (-\infty,z] \big) \qquad \text{for $z\in\R$.} \]

\subsubsection{Transition measure for \conti diagrams as a homeomorphism}

Kerov (see \cite{Kerov1993} and \cite[Chapter~4, Section~1]{KerovBook}) proved
that (for each constant $C>0$) the map
\begin{equation}
	\label{eq:the-map}
	\omega \mapsto \mu_{\omega} 
\end{equation}
is a homeomorphism between the following two topological spaces:
\begin{itemize}
	\item
	the set of \conti diagrams $\omega$ with the property that 
	\[ \omega(z) = |z| \qquad \text{for each $z\notin [-C,C]$},\] 
	equipped with the topology given by the
	supremum distance, and 
	\item the set of centered probability measures with support contained in the interval
	$[-C,C]$, equipped with the weak topology of probability measures.
\end{itemize}

The fact that the map \eqref{eq:the-map} is a homeomorphism allows to retrieve
the transition measure of an arbitrary \conti diagram by approximating it by
a sequence of zigzag diagrams and then by taking the limit of their transition
measures. We will use this idea later on.

\subsection{The main problem: what is the modulus of continuity of the
	transition measure?}

In the current paper we address the problem which can be stated informally as
follows.
\begin{problem}[The main problem]
	\label{problem} 
	
	What is the modulus of continuity of the map \eqref{eq:the-map}?
	
	Given two \conti diagrams $\omega_1$ and $\omega_2$ such that their
	distance (with respect to the supremum distance) is `small', how small is the
	distance between the cumulative functions $K_{\omega_1}$ and $K_{\omega_2}$ (in
	some unspecified metric)?
\end{problem} 

Kerov's result
that \eqref{eq:the-map} is a homeomorphism implies that the distance between
$K_{\omega_1}$ and $K_{\omega_2}$ converges to zero (in some metric which
corresponds to the weak convergence of probability measures) as the distance
between $\omega_1$ and $\omega_2$ tends to zero. However, Kerov's original
proof \cite{Kerov1993} is not very constructive in the sense that it does not
provide much information about the speed of the convergence, and hence it does
not provide quantitative information about the `discrepancy' between the
cumulative functions $K_{\omega_1}$ and $K_{\omega_2}$.

\subsection{The motivations: Schensted row insertion applied to random input} 
\label{sec:marciniak-sniady}

Our motivation for considering Problem~\ref{problem} is related to our recent joint
work with Marciniak \cite{MarciniakSniady} about the statistical properties of
the \emph{Schensted row insertion} applied to random tableaux. We recall that
the Schensted row insertion is the basic component of the
Robinson--Schensted--Knuth algorithm \cite{Fulton1997}.

Up to some minor simplifications, for a given large Young diagram~$\lambda$ we
consider a uniformly random \emph{standard Young tableau} $T$ with the prescribed
shape $\lambda$. We ask about the position of the new box created by the row
insertion $T\leftarrow \mathbf{z}$, when a deterministic number $\mathbf{z}$ is
inserted into $T$. The position of the new box is clearly random, located in
one of the concave corners of $\lambda$. The statistical properties of its
probability distribution can be expressed in terms of the transition measure
$\mu_\lambda$ of the diagram $\lambda$.

The problem becomes even more interesting asymptotically, when the single Young
diagram $\lambda$ is replaced by a \emph{sequence} of Young
diagrams~$(\lambda^{(n)})$. In this setup one can ask whether the (suitably
rescaled) fluctuations of the position of the new box around the mean value
converge to some limit distribution. As an extra bonus, the Young
diagram~$\lambda^{(n)}$ may be chosen to be deterministic or random.

The main result of the paper \cite{MarciniakSniady} is that the fluctuations of
the position of the new box indeed converge to an explicit Gaussian
distribution, provided that the cumulative function of the rescaled diagram
$\lambda^{(n)}$
\[ z \mapsto K_{\left(
	\frac{1}{\sqrt{n}}\ \omega_{\lambda^{(n)}}
	\right)}(z) = 
K_{\lambda^{(n)}}\left( \sqrt{n} \ z \right) \] 
converges in probability to a deterministic limit which is sufficiently smooth.
Additionally, the rate of convergence must be fast enough, more specifically the
error must be at most $o\big( n^{- \frac{1}{4} } \big)$.

\medskip

The goal of the current paper is to provide tools for producing explicit natural
examples of large (deterministic or random) Young diagrams $\lambda$ for which
the cumulative function $K_\lambda$ can be written as a sum of a known smooth
function and a small error term with an explicit asymptotic bound. For such
examples the results of \cite{MarciniakSniady} are applicable. We will come back
to this topic in Section~\ref{sec:applicationsOLD}.

\subsection{The metric $\dxy$ on the set of \conti diagrams} 

For the purposes of this section we will use the following notational
shorthand. For real numbers $x,y\geq 0$ and a \conti diagram $\omega\colon\R \to\R_+$ we
will write $(x,y) \in \omega$ if the point $(x,y)$ belongs to the graph of
$\omega$ \emph{drawn in the French coordinate system} or, equivalently, if
\[   \omega(x-y) = x+y, \]
cf.~\eqref{eq:russian}.  Additionally, for $x\geq 0$ we denote by 
\[ \Pi_Y^\omega(x) = \big\{ y : (x,y) \in \omega \bigskip\} \subseteq \R_+ \]
the projection on the $y$-axis of the intersection of the plot of $\omega$ with
the line having a specified $x$-coordinate. One can verify that
$\Pi_Y^\omega(x)$ is a non-empty closed set. Similarly, for $y\geq 0$ we denote
by
\[ \Pi_X^\omega(y) = \big\{ x : (x,y) \in \omega \bigskip\} \subseteq \R_+\]
the projection on the $x$-axis.

If $\omega_1,\omega_2$ are \conti diagrams, we define their $y$-distance
\[ d_Y(\omega_1, \omega_2) =  
\sup_{x\geq 0} d_H\big( \Pi_Y^{\omega_1}(x), \Pi_Y^{\omega_2}(x) \big)
\]
as the supremum (over all choices of the $x$-coordinate) of the \emph{Hausdorff
	distance} $d_H$ between their $y$-projections. Heuristically, the metric
$d_Y$
is a way to quantify the discrepancy between \conti diagrams along the
$y$-coordinate. In a similar manner we define the $x$-distance between
\conti diagrams as
\[ d_X(\omega_1, \omega_2) =  
\sup_{y\geq 0} d_H\big( \Pi_X^{\omega_1}(y), \Pi_X^{\omega_2}(y) \big).
\]

Finally, we define the distance between $\omega_1$ and $\omega_2$ as the
maximum of the $x$- and the $y$-distance:
\[ \dxy(\omega_1,\omega_2) = \max\big( d_X(\omega_1,\omega_2),
d_Y(\omega_1,\omega_2) \big). \]

\subsection{The main result: modulus of continuity}

The following result provides an answer for Problem~\ref{problem}.

\begin{theorem}[The main result]
	\label{theo:main-modulus}
	Let $\bigomega\colon\R \to \R_+$ be a \conti diagram, let 
	$a<a_0<b_0<b$ be such that:
	\begin{enumerate}[label=(\roman*)]
		\item \label{enum:Lipschitz} the function $\bigomega$ restricted to the
		interval $[a,b]$ is a contraction (i.e., it is $(1-\delta)$-Lipschitz for some
		constant $\delta>0$);
		
		\item the transition measure $\mu_{\bigomega}$ restricted to the interval
		$[a,b]$ is absolutely continuous and has a density which is bounded from above
		by some constant.
	\end{enumerate}
	
	Then there exists a contant $C>0$ with the property 
	that for each $\epsilon\in \left( 0, \frac{1}{2} \right)$ and 
	any \conti diagram
	$\omega$ such that $\dxy(\bigomega, \omega) \leq \epsilon$ the following bound holds true:
	\begin{equation}
		\label{eq:main-theorem}
		 \sup_{z\in [a_0,b_0] } \big| K_{\omega}(z) - K_{\bigomega}(z) \big|  \leq  
	C \epsilon \log \frac{1}{\epsilon}.     
	\end{equation}
\end{theorem}

The proof is postponed to Section~\ref{sec:proof-of-main-result}.

\subsection{The main result, version with a sequence of diagrams}
\label{sec:applicationsOLD}

The following result (for the exponent $\alpha=\frac{1}{4}$) provides a wide
class of examples for the problem which we discussed in
Section~\ref{sec:marciniak-sniady}. Informally speaking, it says that if $\Omega$ is
locally a \emph{`nice'} \conti diagram and a sequence $(\omega_n)$ converges
to $\Omega$ in the metric~$\dxy$ with some prescribed speed then the cumulative
functions of~$(\omega_n)$ converge uniformly to the cumulative function of
$\Omega$ with almost the same speed.

\begin{theorem}[The main result, version for a sequence of diagrams]
	\label{theo:main-for-sequence}
	Let the assumptions of Theorem~\ref{theo:main-modulus} be fulfilled.
	Let $\alpha>0$ be fixed. 	
	Let $(\omega_n)$ be a sequence of \conti diagrams such that 
	\[  \dxy(\bigomega, \omega_n) = o\left(  \frac{1}{  n^{\alpha} \log n}  \right)
	.\]    
	
	\smallskip
	
	Then
	\[ \sup_{z\in [a_0,b_0] } \big| K_{\omega_n}(z) - K_{\bigomega}(z) \big|  = 
	o\left( \frac{1}{n^\alpha} \right). \]    
\end{theorem}

This result follows directly from Theorem~\ref{theo:main-modulus} by setting $\omega:= \omega_n$
and $\epsilon:= \dxy(\bigomega,\omega_n)$. 

While Theorem~\ref{theo:main-modulus} provides more explicit characterization of the
modulus of continuity, both theorems offer comparable asymptotic power. In our
subsequent applications, we will primarily reference
Theorem~\ref{theo:main-for-sequence} for its elegant formulation of the
asymptotic behavior.

\subsection{Toy example} 

In this section we introduce a toy example to which the main result
(Theorem~\ref{theo:main-for-sequence}) applies. Although this example is straightforward
and elementary, allowing us to calculate the transition measures explicitly, it
is still valuable to see how the general tool applies to a specific context.

\subsubsection{The staircase diagrams}

\newcommand{\nklatki}{n}

In the context of Theorem~\ref{theo:main-for-sequence} we will pass to a subsequence
defined as
\[ \nklatki={\nklatki}_N = 1+\cdots+N= \frac{N(N+1)}{2}. \]
We define
\[
	\lambda^{(\nklatki_N)}= (N,N-1,\dots,3,2,1) 
\]
to be a \emph{staircase Young diagram} with $\nklatki_N$ boxes. 
We define 
\[ \omega_n = \omega_{\frac{1}{\sqrt{n}} \lambda^{(n)}  } \]
to be the rescaled profile of the staircase diagram.

We define the \emph{triangle diagram}
\[ \Omega(z) = 
\begin{cases} 
	\sqrt{2} & \text{for $|z|\leq \sqrt{2}$}, \\
	|z|      & \text{for $|z|\geq \sqrt{2}$}.
\end{cases}
\]
Then the sequence $(\omega_n)$ converges to $\Omega$ in the metric $\dxy$, with the
rate of convergence given by
\begin{equation}
	\label{eq:rate-of-convergence}
	\dxy\big( \omega_n, \Omega \big) = 
	O\left( \frac{1}{N}  \right) = 
	O\left( \frac{1}{\sqrt{n}}  \right). 
\end{equation}

\subsubsection{The transition measures}

A straightforward calculation of the residues of the Cauchy transform shows that
the transition measure of $\lambda^{(\nklatki)}$ is supported on the set of
even, respectively odd integer numbers
\[ \{-N, -N+2, \dots, N-2, N\} \] 
with the probabilities
\[
	\mu_{\lambda^{(n)}}(2k-N) = \frac{1}{2^{2N}} \binom{2k}{k}
	\binom{2N-2k}{N-k}
	\qquad \text{for $k\in\{0,\dots,N\}$.} 
\]
This probability distribution appears naturally in the
context of random walks and the \emph{arcsine theorem}, 
see \cite[Chapter~III]{FellerVol1}.  

On the other hand, 
the transition measure $\mu_{\Omega}$
of the triangle diagram is the arcsine law supported on the interval $I=\big[
-\sqrt{2}, \sqrt{2} \big]$ with the density
\[ f_{\AS}(z) = \frac{1}{\pi \sqrt{2-z^2} } \qquad \text{for $z\in I$}.\]

\subsubsection{Application of Theorem~\ref{theo:main-for-sequence}}

By \eqref{eq:rate-of-convergence} it follows that the assumptions of
Theorem~\ref{theo:main-for-sequence} are fulfilled for the exponent $\alpha= \frac{1}{2} -
\epsilon$, for any $\epsilon>0$, and for any quadruple $a<a_0<b_0<b$ from the
open interval $( - \sqrt{2}, \sqrt{2}) $. As a consequence, the rate of
convergence of the cumulative function of the rescaled staircase diagram
$\omega_n$ towards the cumulative function of the triangle diagram $\Omega$ is
bounded from above by
\begin{equation}
	\label{eq:rate-AS}
	\sup_{z\in [a_0,b_0] } \big| K_{\omega_n}(z) - K_{\bigomega}(z) \big|  = 
	o\left( \frac{1}{n^{\frac{1}{2}-\epsilon}} \right) 
	= o\left( \frac{1}{N^{1-2 \epsilon}} \right)
\end{equation} 
for each $\epsilon>0$.

\subsubsection{What is the optimal rate of convergence?}

A straightforward calculation based on the Stirling approximation provides an
upper bound of the form
\begin{equation}
	\label{eq:rate-AS2}
	\sup_{z\in [a_0,b_0] } \big| K_{\omega_n}(z) - K_{\bigomega}(z) \big|  = 
	O\left( \frac{1}{\sqrt{n}} \right) 
	= O\left( \frac{1}{N} \right).
\end{equation} 
On the other hand, the transition measure
of $\omega_n$ is a discrete measure which in the interval $(a_0, b_0)$ has
atoms which are of order $\Theta\left( \frac{1}{N} \right)$. It follows that
also the left-hand side of \eqref{eq:rate-AS2} is \emph{at least} of order
$\Theta\left( \frac{1}{N} \right)= \Theta\big( \frac{1}{\sqrt{n}} 
\big)$, hence the bound \eqref{eq:rate-AS2} is optimal.

\medskip

We can see that Theorem~\ref{theo:main-for-sequence} provided us with an upper bound
\eqref{eq:rate-AS} which is very close to the optimal result
\eqref{eq:rate-AS2}. We conclude that there is not much room for
improvement left in Theorem~\ref{theo:main-for-sequence}

\subsection{The main result for a sequence of random \conti diagrams}

Theorem~\ref{theo:main-for-sequence} allows the following extension for random diagrams.

\begin{theorem}[The main result, version for a sequence of random diagrams]
	\label{theo:main-random} We keep the assumptions from
	Theorem~\ref{theo:main-for-sequence}, but now $(\omega_n)$ is a sequence of \emph{random}
	\conti diagrams such that
	\[ n^{\alpha} \log n \cdot \dxy(\bigomega, \omega_n) \xrightarrow[n\to\infty]{P} 0.\]    
	
	\smallskip
	
	Then
	\[ n^\alpha \sup_{z\in [a_0,b_0] } \big| K_{\omega_n}(z) - K_{\bigomega}(z) \big|  \xrightarrow[n\to\infty]{P}  0.\]    
\end{theorem}
Again, this result follows directly from Theorem~\ref{theo:main-modulus}.

\subsection{Content of the paper}

We start in Section~\ref{sec:comparison} with a simple monotonicity result
(Proposition~\ref{prop:comparison-of-transition}) which, roughly speaking, says that when
one \emph{`adds new boxes'} to a \conti diagram on the right-hand side, or
\emph{`removes some boxes'} on the left-hand side, the values of the cumulative
function in the middle can only increase. Thanks to this monotonicity result we
will later estimate the cumulative function of generic diagrams in terms of the
cumulative function of some \emph{nice} \conti diagram over which we have
a good control.

\smallskip

In Sections~\ref{sec:shifted} and \ref{sec:shifted-measure} we address the following problem.
In the metric space of the \conti diagrams (equipped with the metric $\dxy$) we
consider the ball $B_{\epsilon,\Omega}$ which has radius $\epsilon>0$ and is
centered in some \conti diagram $\Omega$. Inside this ball we consider a subset
\begin{equation}
	\label{eq:ball}
	\big\{  \omega\in B_{\epsilon,\Omega} : \omega(z_0) = v \big\} 
\end{equation}
of diagrams which at given $z_0\in\R$ take a specified value $v$. \emph{What can
	we say about the supremum of the map $\omega \mapsto K_{\omega}(z_0)$ over this
	set \eqref{eq:ball}?}

In Section~\ref{sec:shifted} we will construct a diagram $\omega$ in \eqref{eq:ball}
for which the cumulative function $K_{\omega}(z_0)$ takes the maximal value. 
As we shall see, this diagram is given by the \emph{$\epsilon$-shift} of~$\Omega$ along some specific affine function~$f$, cf.~Proposition~\ref{prop:steepest}.

In Section~\ref{sec:shifted-measure} we will find an upper bound for the supremum 
(Proposition~\ref{prop:tails}) in terms of~$\bigomega$ and its transition measure.

\smallskip

In Section~\ref{sec:bounds} we prove a twin pair of results:
Theorem~\ref{theo:module-continuity-A} and Theorem~\ref{theo:module-continuity-B}. 
Together they give
an upper bound for the difference 
\begin{equation}
	\label{eq:difference}
	\big| K_{\Omega}(z_0)- K_{\omega}(z_0)\big|
\end{equation}
of the cumulative functions of two \conti diagrams $\Omega$ and $\omega$.
The roles played by these diagrams are quite different: in the applications
$\Omega$ is \emph{`nice'} (for example, $\Omega\colon \R\to\R_+$ may be given by
an explicit analytic expression, additionally it may have a derivative which
locally is away from $\pm
1$, etc.), while $\omega$ is quite generic and we not have much information
about it. The aforementioned upper bound for \eqref{eq:difference} is
expressed in terms of the distance $\dxy(\Omega,\omega)$ between the diagrams, as
well as in terms of the shape of $\Omega$, and the transition measure of
$\Omega$. Notably, this bound does not depend on some fine details of~$\omega$.

\smallskip

Finally, in Section~\ref{sec:proof-of-main-result} we apply
Theorems~\ref{theo:module-continuity-A} and \ref{theo:module-continuity-B} in order to prove
Theorem~\ref{theo:main-modulus}.

\section{Steeper diagram has a larger cumulative function}
\label{sec:comparison}

\begin{proposition}
	\label{prop:comparison-of-transition}
	Let $\omega_1,\omega_2 \colon \R\to\R_+$ be two \conti diagrams and let
	$z_0\in\R$ be fixed.
	Assume that the following two conditions hold true:
	\begin{equation}
		\label{eq:tilt}
		\left\{ 
		\begin{aligned}
			\omega_1(z) & \geq \omega_2(z) \qquad \text{for each $z\leq z_0$}, \\
			\omega_1(z) & \leq \omega_2(z) \qquad \text{for each $z\geq z_0$}.
		\end{aligned}
		\right.
	\end{equation}

	\smallskip
	
		Then the left limit of the cumulative function of diagram $\omega_1$
		and the value of the cumulative function of diagram $\omega_2$
		at point $z_0$ satisfy
	\begin{equation}
		\label{eq:cdf-mono}
		\lim_{\tau \to 0^+} K_{\omega_1}(z_0 - \tau) \leq
		K_{\omega_2}(z_0). 
	\end{equation}
\end{proposition}
The remaining part of this section is devoted to the proof.

\subsection{Special case: Young diagrams}
We start with the special case when $\omega_1= \omega_{\lambda^1}$ and
$\omega_2=\omega_{\lambda^2}$ are profiles of some Young diagrams $\lambda^1$
and $\lambda^2$. 
We will show a stronger result
\begin{equation}
	\label{eq:cdf-mono-a}
	K_{\omega_1}(z_0 ) \leq K_{\omega_2}(z_0). 
\end{equation}

The
assumption \eqref{eq:tilt} implies that the diagram $\lambda^1$ can be
transformed to~$\lambda^2$ by a
sequence of \emph{elementary changes}. Each elementary change is either: 
\begin{enumerate}[label=(\alph*)]
	\item \label{item:elementary-add} 
	an \emph{addition of a box} with the $u$-coordinate of the two central
	vertices equal to $\xi\in \Z$ with $\xi\geq z_0+1$, see Figure~\ref{fig:newbox},
	or
	\item \label{item:elementary-remove} 
	a \emph{removal of a box} with the $u$-coordinate of the two central
	vertices equal to $\xi \in \Z$ with \mbox{$\xi\leq z_0-1$}.
\end{enumerate}
Each neighboring pair of Young diagrams along this sequence of elementary
changes fulfills an analogue of the inequalities \eqref{eq:tilt} and we will
show below that an analogue of \eqref{eq:cdf-mono-a} holds true for each such a 
pair.
By chaining the sequence of the inequalities of the form \eqref{eq:cdf-mono-a}
along the sequence of the elementary changes, the desired inequality for the
endpoints $\omega_1$ and $\omega_2$ would follow.

\begin{figure}
	\begin{tikzpicture}[scale=1,rotate=45]
		\coordinate (p0) at (3,0);
		\coordinate (p1) at (3,1);
		\coordinate (p2) at (2,1);
		\coordinate (p3) at (2,4);
		\coordinate (p4) at (0,4);
		\coordinate (p5) at (0,6);
		\draw[blue,fill=blue!20] (p2) rectangle +(1,1);
		\begin{scope}
			\clip  (p0) -- (p1) -- (p2) -- (p3) -- (p4) -- (0,0);
			\draw[black!50] (0,0) grid (10,10);
		\end{scope}
		\draw[ultra thick] (7,0) -- (p0) -- (p1) -- (p2) -- (p3) -- (p4) --
		(p5);
		\draw[->] (-4,4) -- (4,-4) node[anchor=west] {$u$};
		
		\draw[dotted,blue,very thick] (p2) -- (1/2,-1/2) node[anchor=north]
		{\textcolor{black}{\small $\xi$}};
		
		\foreach \u in {-7,...,7}
		{\draw (\u/2,-\u/2)  +(-1pt,-1pt) -- +(1pt,1pt); }
	\end{tikzpicture}
	\caption{A Young diagram $\lambda$ (white boxes) and its profile
		$\omega_{\lambda}$ (the thick zigzag line). The $u$-coordinate of the
		two
		central vertices of the blue box is equal to~$\xi$.} \label{fig:newbox}
\end{figure}
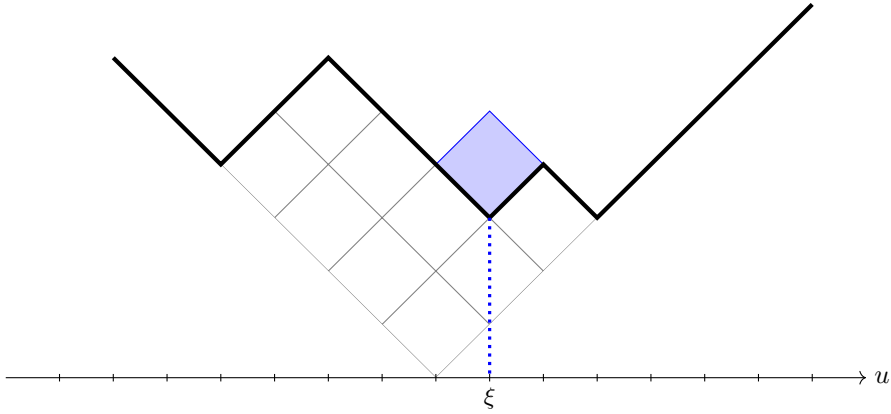

\smallskip

We proceed now with an elementary change of type \ref{item:elementary-add}. Let
$\omega=\omega_{\lambda}$ be a profile of a Young diagram and let
$\overline{\omega}$ be the profile of the diagram obtained from $\omega$ by
adding a single box, as in Figure~\ref{fig:newbox}. The elementary change transforms
the profile $\omega$ to $\overline{\omega}$. By comparing the sets of convex
and the concave corners for these two profiles it follows that their Cauchy
transforms are related by the following equality:
\[     \cauchy_{\overline{\omega}}(z) =  \frac{(z-\xi)^2}{(z-\xi-1)(z-\xi+1)} 
\cauchy_\omega(z). \]
It follows that the atoms of the transition measures of $\omega$ and
$\overline{\omega}$ are related as follows:
\begin{itemize}
	\item for any integer $z\notin\{ \xi-1,\ \xi,\ \xi+1 \}$ 
	\begin{multline}
		\label{eq:addbox-away}
		\mu_{\overline{\omega}}\big( \{ z \} \big)  =
		\frac{(z-\xi)^2}{(z-\xi-1)(z-\xi+1)} \mu_{\omega}\big( \{ z \} \big)   \\ =
		\left[
		1+ \frac{1}{(z-\xi)^2-1} \right] \mu_{\omega}\big( \{ z \} \big) \geq 
		\mu_{\omega}\big( \{ z \} \big);
	\end{multline}
	\item additionally, the transition measure $\mu_{\overline{\omega}}$ 
	potentially may have atoms in $\xi\pm 1$ while the transition measure
	$\mu_{\omega}$
	does not have any atoms there, so the inequality     
	\begin{equation}
		\label{eq:addbox-close}
		\mu_{\overline{\omega}}\big( \{ z \} \big)  \geq \mu_{\omega}\big( \{ z
		\}
		\big)
	\end{equation}
	holds true also for $z\in \{ \xi-1,\ \xi+1 \}$.
\end{itemize}
Conclusion: for each integer $z\neq \xi$ we proved the inequality \eqref{eq:addbox-close}.

The cumulative distribution functions can be expressed as sums of the
appropriate atoms, hence
\[ K_{\overline{\omega}}(z_0) = 
\sum_{\substack{z \in \Z \\ z \leq z_0}} \mu_{\overline{\omega}}\big(\{z\} \big)	
\geq  
\sum_{\substack{z \in \Z \\ z \leq z_0}} \mu_{{\omega}} \big(\{z\} \big) 
= K_{{\omega}}(z_0)
\]
where we used the inequalities \eqref{eq:addbox-away} and
\eqref{eq:addbox-close}, as well as the fact that if the integer $z$
contributes to the above sums then $z \leq \xi-1$. This concludes the proof for
an elementary change of type \ref{item:elementary-add}.

\smallskip

We proceed now with an elementary change of type \ref{item:elementary-remove}.
We keep the above notations; the elementary change transforms the \conti
diagram $\overline{\omega}$ to $\omega$. The cumulative distribution functions
can be expressed in terms of the tails of the respective distributions
therefore
\[ K_{{\omega}}(z_0) = 
1- \sum_{\substack{z \in \Z \\ z > z_0}} \mu_{{\omega}}\big(\{z\} \big)  \geq 
1- \sum_{\substack{z \in \Z \\ z > z_0}} \mu_{\overline{\omega}}\big(\{z\}
\big)  
= K_{\overline{\omega}}(z_0)
\]
where we used the inequalities \eqref{eq:addbox-away} and
\eqref{eq:addbox-close}, as well as the fact that if the integer $z$
contributes to the above sums then $z \geq \xi+1$. This concludes the proof for
an elementary change of type \ref{item:elementary-remove}.

\subsection{Rescaled Young diagrams} 

For any $c>0$ and any Young diagram
$\lambda$, the transition measure of the rescaled profile 
\[ \omega_{ c \lambda}(z) = \frac{1}{c}\ \omega(c z) \]
is a dilation of the transition of the original Young diagram $\lambda$. As a
consequence, the inequality \eqref{eq:cdf-mono-a} holds true also when
$\omega_1= \omega_{c
	\lambda^1}$ and $\omega_2=\omega_{c \lambda^2}$ are rescaled profiles of two
Young diagrams.

\subsection{The general case} 
\label{sec:comparison-proof-general}

For each $i\in\{1,2\}$ and an integer $n\geq
1$ we define~$\lambda^{i,n}$ to be the largest Young diagram $\lambda$ with the
property that
\[ \omega_{ \frac{1}{n} \lambda }(z) \leq \omega_i(z) \]
holds true for each $z\in\R$.
The pair of rescaled profiles 
\[ \omega_{1,n}:=\omega_{ \frac{1}{n} \lambda^{1,n }}
\quad 
\text{and}
\quad 
\omega_{2,n}:= \omega_{ \frac{1}{n} \lambda^{2,n }}\] 
fulfills an analogue of the system of inequalities \eqref{eq:tilt}, hence the
above discussion is applicable and therefore
\begin{equation}
	\label{eq:k-is-nice}
	K_{\omega_{1,n}}(z_0) \leq K_{\omega_{2,n} }(z_0).
\end{equation}

For each $i\in\{1,2\}$, the sequence of generalized diagrams
$(\omega_{i,n})_{n\geq 1}$
converges to the generalized diagram $\omega_i$ in the suitable topological
space described
by Kerov \cite{Kerov1993}. From the result of Kerov \cite{Kerov1993} it follows
that the sequence of transition measures $(\mu_{\omega_{i,n}})_{n\geq 1}$
converges in the weak topology to the transition measure $\mu_{\omega_i}$. 

Let $z_- < z_0$ be a continuity point of $K_{\omega_1}$ 
and let $z_+> z_0$ be a continuity point of $K_{\omega_2}$.
It follows that
\begin{multline*} 
	K_{\omega_1}(z_-)= \lim_{n\to\infty} K_{\omega_{1,n}} (z_-) \leq 
	\liminf_{n\to\infty} K_{\omega_{1,n}}(z_0) \\ \leq  
	\liminf_{n\to\infty} K_{\omega_{2,n}}(z_0) 
	\leq \lim_{n\to\infty} K_{\omega_{2,n}}(z_+) =
	K_{\omega_2}(z_+).
\end{multline*}
Since the set of the continuity points of $K_{\omega_1}$ and $K_{\omega_2}$ is
dense, we may consider the left limit of the left-hand side at $z_- = z_0$ and
the right limit of the right-hand side at $z_+= z_0$. As a consequence,
\[ \lim_{\tau\to 0} K_{\omega_{1}}(z_0-\tau) \leq
K_{\omega_2}(z_0), \]
as required. This completes the proof of Proposition~\ref{prop:comparison-of-transition}.

\section{Shifted diagrams}
\label{sec:shifted}

\subsection{The input data}
\label{sec:input}

Let $\Omega\colon \R\to\R_+$ be a \conti diagram, let $\epsilon>0$ be a real number and let $f\colon \R \to\R$ be an affine function of the form
\[ f(x) = x + b\]
for some constant $b>0$. For this input data we will construct in the current
section a new \conti diagram
$\referenceomega=\referenceomega_{\Omega,\epsilon,f}$, called
\emph{$\epsilon$-shift of $\Omega$ along $f$}.

\begin{figure}
	
	\begin{tikzpicture}[scale=0.3, rotate=45]

		\draw[Green!30, line width=5 pt] (0,20) 
		-- ++(0,-8) 
		-- ++(3,0) -- ++(0, -2) node[anchor=north,opacity=1,black]
		{$\referenceomega$}
		-- ++(5,0) 
		-- ++(0, -4) 
		-- ++(11,0) 
		-- ++(0, -3) 
		-- ++(2,0) node[anchor=south,opacity=1,black] {$\referenceomega$}
		-- ++(0, -3) 
		-- ++(4,0) 
		; 
		
		\draw[->, thick, black!50]  (-12,12) -- (15,-15) node[anchor=
		west,black]{$u$};
		\draw[->, thick, black!50]  (-2,-2) -- (15,15) node[anchor=
		south,black]{$v$};
		\draw[thin, black!50] (0,20) -- (0,0) -- (20,0);

		\draw[thick, densely dotted] (-4,6) -- +(32,0) node[anchor=west]
		{\textcolor{black}{$f$}};

		\draw[ultra thick] (0,24) 
		-- ++(0,-8) 
		-- ++(3,0) -- ++(0, -2) 
		-- ++(5,0) -- ++(0, -6) 
		-- ++(7,0) -- ++(0, -5) 
		-- ++(2,0) -- ++(0, -3) 
		-- ++(12,0)  node[anchor=west] {\textcolor{black}{$\bigomega$}};
		;

		\draw[blue, dashed,  thick] (0,20) 
		-- ++(0,-8) 
		-- ++(3,0) -- ++(0, -2) 
		-- ++(5,0) -- ++(0, -6) 
		-- ++(7,0) -- ++(0, -5) 
		-- ++(2,0) -- ++(0, -3) 
		-- ++(10,0) 
		node[anchor=south] {\textcolor{blue}{blue}};    
		; 
		
		\draw (4,24) node[anchor=south] {\textcolor{red}{red}};
		\draw[red, dashed, thick] (4,24)
		-- ++(0,-8) 
		-- ++(3,0) -- ++(0, -2) 
		-- ++(5,0) -- ++(0, -6) 
		-- ++(7,0) -- ++(0, -5) 
		-- ++(2,0) -- ++(0, -3) 
		-- ++(3,0) 
		;

		\draw[red,dotted,thick, ->] (0,22) -- +(4,0);
		\draw[blue,dotted,thick, ->] (24,0) -- +(0,-4);

		\fill (8,6) circle (0.3);
		\fill (19,6) circle (0.3);
		
		\draw[dashed] (8,6) -- (1,-1);
		\draw[thick] (1,-1) +(0.2,0.2) -- +(-0.2,-0.2) node[anchor=north]
		{$z_-$};
		\draw[dashed] (19,6) -- (6.5,-6.5);
		\draw[thick] (6.5,-6.5) +(0.2,0.2) -- +(-0.2,-0.2) node[anchor=north]
		{$z_+$};

	\end{tikzpicture}
	\caption{The solid black zigzag line is the \conti diagram $\bigomega$.
		The upper
		red dashed zigzag line was obtained by translation of $\bigomega$ by
		the vector
		(in the French coordinates) $x=\epsilon$, $y=0$. The bottom blue dashed
		zigzag line was
		obtained by translation of $\bigomega$ by the vector $x=0$,
		$y=-\epsilon$. 
		The red and the blue dotted arrows indicate these translation vectors. 
		\\
		The densely dotted diagonal
		black line is the plot of the affine function $f$. 
		The two black circles indicate the intersection points of the function
		$f$
		with, respectively, the blue and the red zigzag lines. Their
		$u$-coordinates are equal to, respectively, $z_-$ and $z_+$. 
		\\ 
		The thick green zigzag line is the \conti diagram $\referenceomega$
		which was obtained by combining a segment of the blue dashed line 
		(to the left of the first black circle), a segment of the function~$f$
		(between the black circles), and a segment of the red dashed line 
		(to the right of the second black circle). }
	\label{fig:three-zigzags}
\end{figure}
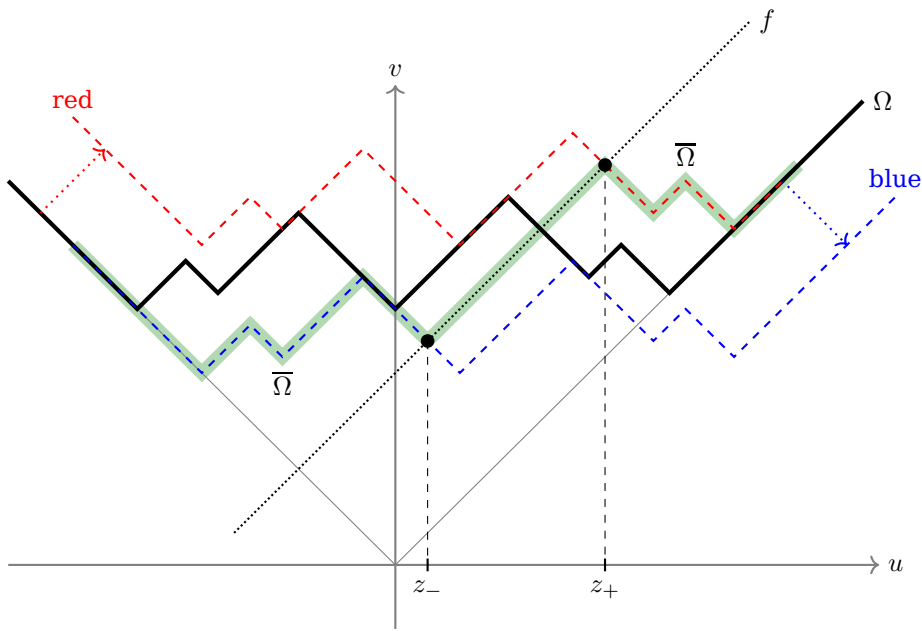

\subsection{The upper and the lower bound for the $\epsilon$-ball}
\label{sec:red-and-blue}

Let $\smallomega$ be an arbitrary \conti diagram such that $\dxy(\bigomega,
\smallomega) \leq \epsilon$. Since $\epsilon$ is an upper bound on the deviation
between the curves $\bigomega$ and $\smallomega$ along the $x$-coordinate, it
follows that
\begin{equation} 
	\label{eq:red-bound}
	\smallomega(z) \leq \bigomega(z-\epsilon) + \epsilon 
\end{equation}
holds true for any $z\in \R$. The function which appears on the right hand side,
i.e.,
\[ z\mapsto \bigomega(z-\epsilon) + \epsilon \]
is depicted in Figure~\ref{fig:three-zigzags} as the red dashed line. Analogously,
since $\epsilon$ is an upper bound on the deviation between the curves
$\bigomega$ and $\smallomega$ along the $y$-coordinate, it follows that
\begin{equation} 
	\label{eq:blue-bound} 
	\smallomega(z) \geq \bigomega(z-\epsilon) - \epsilon 
\end{equation}
holds true for any $z\in \R$.
The function which appears on the right hand side, i.e.,
\[ z\mapsto \bigomega(z-\epsilon) - \epsilon \]
is depicted in Figure~\ref{fig:three-zigzags} as the blue dashed line. In this way
any \conti diagram within the $\epsilon$-ball around $\bigomega$ (with
respect to the metric $\dxy$) lies within the area between the red and the blue
zigzag lines.

\subsection{The intersection points $z_-$ and $z_+$}
\label{sec:intersection}

Let $z_-$ be the minimal number such that 
\begin{equation}  
	\label{eq:zminus}
	f(z_-) = \bigomega(z_- -\epsilon)- \epsilon 
\end{equation}
and  let $z_+$ be the maximal number such that 
\begin{equation} 
	\label{eq:zetplus}
	f(z_+) = \bigomega(z_+ -\epsilon) + \epsilon. 
\end{equation}
With the notations of Figure~\ref{fig:three-zigzags} these are the $u$-coordinates of
the intersection of the dotted line~$f$ with the blue zigzag line and with the
red zigzag line. Clearly, the intersection points exist and the numbers $z_-$,
$z_+$ are well-defined and fulfill
\[ z_- \leq z_+,\]
see Figure~\ref{fig:three-zigzags}.

	\subsection{The shifted diagram}
	\label{sec:reference-diagram}
	
	We define a \conti diagram $\referenceomega=\referenceomega_{\Omega,\epsilon,f}$ given by
	\[ \referenceomega(z) = \begin{cases} 
		\bigomega(z-\epsilon) - \epsilon  & \text{for $z\leq z_-$}, \\
		f(z) & \text{for $z_-\leq z \leq z_+$}, \\
		\bigomega(z-\epsilon) + \epsilon & \text{for $z\geq z_+$}.
	\end{cases}
	\]
	Equivalently,
	\begin{equation}
		\label{eq:minmax}
		\referenceomega(z) =  
		\min\bigg[ \max\Big[ \bigomega(z-\epsilon) - \epsilon, \quad
		f(z) \Big], \quad
		\bigomega(z-\epsilon) + \epsilon \bigg].
	\end{equation}
	We refer to $\referenceomega$ as \emph{$\epsilon$-shift of $\Omega$ along $f$}.
	In Figure~\ref{fig:three-zigzags} this diagram is shown as the thick green zigzag
	line; it is a piece-wise combination of the red zigzag line, the dotted line,
	and the blue zigzag line. 
	
	\medskip
	
	This definition has the following heuristic interpretation in the special case
	when $\epsilon$ and $b=f(0)$
	are positive integers, and $\bigomega$ is a profile of some Young diagram $\lambda$. In this case, in the French coordinate system the curve~$f$ is a
	horizontal straight line with the equation $y=b$. Then the shifted diagram $\referenceomega$ is the
	profile of the Young diagram obtained from $\bigomega$ by the following
	procedure:
	\begin{itemize}
		\item we extend each of the rows of $\lambda$ lying below the horizontal 
		line $f$ by $\epsilon$ additional boxes,
		
		\item among the rows which are lying above the horizontal line $f$, we remove
		the bottom $\epsilon$ rows, and we shift the remaining rows down, so that they
		are placed on top of the horizontal line~$f$.
	\end{itemize}
	
	\subsection{Shifted diagram as a reference point for the $\epsilon$-ball}

	\begin{proposition}
		\label{prop:steepest}
		Let $z_0\in\R$, $\epsilon>0$, let $\smallomega$, $\bigomega$ be \conti
		diagrams such that $\dxy(\smallomega,\bigomega)\leq \epsilon$ and such that
		$\smallomega(z_0)> z_0$. Let
		\begin{equation}
			\label{eq:f-for-z0}
			f(z)  =  z- z_0+  \smallomega(z_0)
		\end{equation}
		be the unique affine function with slope $1$ and such that
		$f(z_0)=\smallomega(z_0)$. Let
		$\referenceomega=\referenceomega_{\Omega,\epsilon,f}$ be the $\epsilon$-shifted
		diagram $\Omega$ along $f$.
		
		\medskip
		
		Then the left limit of the cumulative function of $\omega$ fulfills
		\[ 		\lim_{\tau \to 0^+} K_{\omega}(z_0 - \tau) \leq
		K_{\referenceomega}(z_0). \]
	\end{proposition}
	\begin{proof}
		From the construction of $\referenceomega$ it follows that
		\begin{align*}
			\smallomega(z) & \geq \referenceomega(z) \qquad \text{for each $z\leq
				z_0$},\\
			\smallomega(z) & \leq \referenceomega(z) \qquad \text{for each $z\geq z_0$.}
		\end{align*}
		Proposition~\ref{prop:comparison-of-transition} completes the proof.
	\end{proof}

	\section{Transition measure for shifted diagrams}
	\label{sec:shifted-measure}
	
	The main result of the current section is Proposition~\ref{prop:tails} which provides an
	upper bound for the tail of the transition measure of a shifted diagram.
	Proposition~\ref{prop:tails-zigzag} will serve as a technical intermediate step.

	\subsection{Function $P_{z_+}^{\min}$}	
	In the remaining part of this section $\epsilon>0$ is a fixed constant.
	It will be convenient to treat the dependence of $P_{z_+}^{\min}$ (see below) on
	$\epsilon$ as implicit.
	
	For a real number $z_+$ we define
	\begin{align}
		\label{eq:p-min}
		P_{z_+}^{\min}(z) & = 
		\begin{cases} \displaystyle  
			1 - \frac{\epsilon}{z+\epsilon-z_+} & \text{for $z\geq z_+$}, \\
			0 & \text{otherwise,}
		\end{cases}  \\[2ex]
		\nonumber
		& =
		\begin{cases} \displaystyle  
			\frac{z-z_+}{z+\epsilon-z_+} & \text{for $z\geq z_+$}, \\[2ex]
			0 & \text{otherwise.}
		\end{cases}
	\end{align}
	The function $z\mapsto P_{z_+}^{\min}(z)$ is continuous and weakly increasing;
	it takes values from the interval $[0,1)$.

	\subsection{Lower bound for the transition measure of
		$\referenceomega$, only zigzag diagrams} 
	
	We start with an additional assumption that the \conti diagram $\bigomega$
	is a zigzag, cf.~Section~\ref{sec:zigzag}. In this way we can use the notations from
	Section~\ref{sec:transition-measure}, in particular let
	$\kerx_0<\cdots<\kerx_{\kerell}$ be the $u$-coordinates of the concave corners
	of $\bigomega$ and let $\kery_1<\cdots<\kery_\kerell$ be the convex corners, see
	Figure~\ref{fig:transition}.
	
	\begin{proposition}
		\label{prop:tails-zigzag} 
		
		Let $\bigomega$ be a zigzag diagram, and let
		$\epsilon>0$, $f$ and $z_+$ be as described in
		Sections~\ref{sec:input} and \ref{sec:intersection}, and let $\referenceomega=
		\referenceomega_{\Omega,f,\epsilon}$ be the corresponding shifted diagram.
		
		Then for any non-negative function $\phi\colon \R \to \R_+$	
		\[    \int_{-\infty}^\infty   \phi(z)  \dif\mu_{\referenceomega}(z)  \geq	 
		\int_{-\infty}^\infty  \phi(z+\epsilon)\ P_{z_+}^{\min}(z+\epsilon) \dif \mu_{\bigomega}(z). \]
	\end{proposition}
	
	Above, $\phi$ plays the role of a \emph{test function} and this result can be interpreted informally as inequality 
	\[ \mu_{\referenceomega}(\cdot) \geq  P_{z_+}^{\min}({}\cdot{} )\ \mu_{\bigomega}({}\cdot{} -\epsilon) \]
	between measures.
	
	The remaining part of the current section is devoted to the proof. 
	
	\subsubsection{The Cauchy transform of $\referenceomega$} 
	The shifted convex and the concave corners 
	\[ \kerx_0 + \epsilon< \kery_1+ \epsilon < \kerx_1+ \epsilon < \cdots 
	< \kery_{\kerell}+ \epsilon < \kerx_{\kerell}+ \epsilon\]
	split the real line into a number of (finite or semi-infinite) intervals; the
	number $z_-$ belongs to one of them. Note that the case when $ \kerx_i +
	\epsilon < z_- \leq \kery_{i+1} + \epsilon$ is not possible since this would
	contradict the minimality of $z_-$, see
	Figures~\ref{fig:three-zigzags} and \ref{fig:three-zigzags-simplified}. There exists
	therefore an index $i\in\{0,\dots,\kerell\}$ such that
	\[  \kery_{i} + \epsilon < z_- \leq \kerx_{i} + \epsilon,\]
	see Figure~\ref{fig:three-zigzags-simplified}.
	In the case $i=0$ we use the convention that $\kery_0= - \infty$ so that the
	first inequality is automatically fulfilled. 
	Similarly, there exists an index
	$j\in\{1,\dots,\kerell\}$ such that
	\[  \kery_{j} + \epsilon \leq  z_+ < \kerx_{j} + \epsilon, \]
	see Figure~\ref{fig:three-zigzags-simplified}.
	
	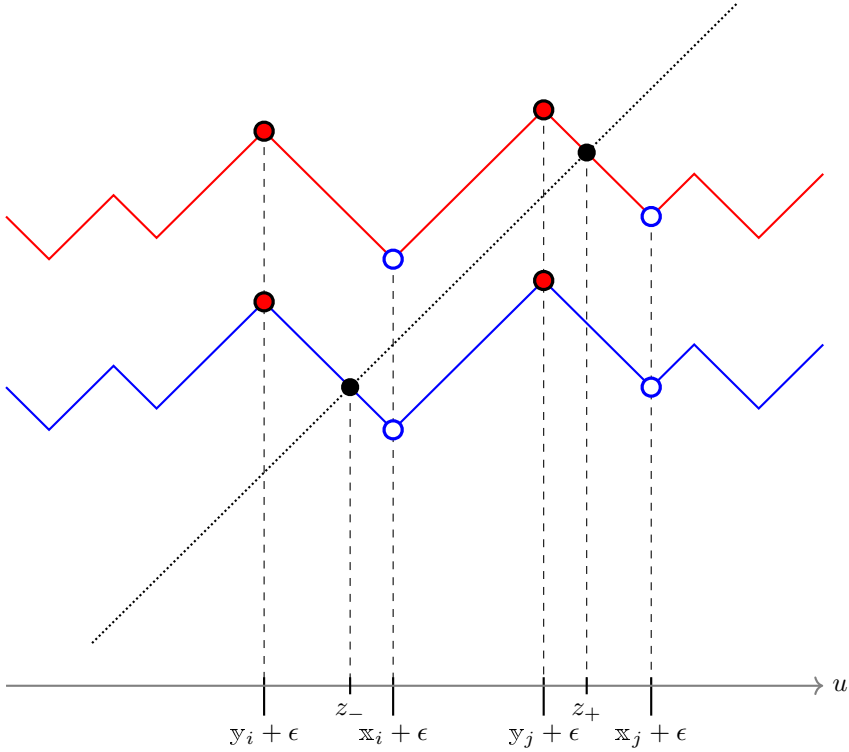
\begin{figure}
		
		\begin{tikzpicture}[scale=0.4, rotate=45]

			\draw[dashed] (12,14) -- (-1,1);
			\draw[thick] (-1,1) +(0.2,0.2) -- +(-0.7,-0.7) node[anchor=north]
			{$\kery_i + \epsilon$};

			\draw[dashed] (12,8) -- (2,-2);
			\draw[thick] (2,-2) +(0.2,0.2) -- +(-0.7,-0.7) node[anchor=north]
			{$\kerx_i + \epsilon$};

			\draw[dashed] (19,8) -- (5.5,-5.5);
			\draw[thick] (5.5,-5.5) +(0.2,0.2) -- +(-0.7,-0.7) node[anchor=north]
			{$\kery_j + \epsilon$};

			\draw[dashed] (19,3) -- (8,-8);
			\draw[thick] (8,-8) +(0.2,0.2) -- +(-0.7,-0.7) node[anchor=north]
			{$\kerx_j + \epsilon$};

			\draw[->, thick, black!50]  (-7,7) -- (12,-12) node[anchor=
			west,black]{$u$};
			\draw[thick, densely dotted] (-4,6) -- +(30,0);

			\draw[blue, thick] (0,14) 
			-- ++(0,-2) 
			-- ++(3,0) -- ++(0, -2) 
			-- ++(5,0) -- ++(0, -6) 
			-- ++(7,0) -- ++(0, -5) 
			-- ++(2,0) -- ++(0, -3) 
			-- ++(3,0) 
			; 

			\draw[red, thick] (4,18)
			-- ++(0,-2) 
			-- ++(3,0) -- ++(0, -2) 
			-- ++(5,0) -- ++(0, -6) 
			-- ++(7,0) -- ++(0, -5) 
			-- ++(2,0) -- ++(0, -3) 
			-- ++(3,0) 
			;

			\fill (8,6) circle (0.3);
			\fill (19,6) circle (0.3);
			
			\draw[dashed] (8,6) -- (1,-1);
			\draw[thick] (1,-1) +(0.2,0.2) -- +(-0.2,-0.2) node[anchor=north]
			{$z_-$};
			\draw[dashed] (19,6) -- (6.5,-6.5);
			\draw[thick] (6.5,-6.5) +(0.2,0.2) -- +(-0.2,-0.2) node[anchor=north]
			{$z_+$};

			\draw[very thick,fill=red] (19,8) circle (0.3);
			\draw[very thick,fill=red] (15,4) circle (0.3);
			
			\draw[very thick,fill=red]  (12,14) circle (0.3);
			\draw[very thick,fill=red]  (8,10) circle (0.3);

			\draw[blue,very thick,fill=white] (19,3) circle (0.3);
			\draw[blue,very thick,fill=white] (15,-1) circle (0.3);

			\draw[blue,very thick,fill=white] (12,8) circle (0.3);
			\draw[blue,very thick,fill=white] (8,4) circle (0.3);

		\end{tikzpicture}
		
		\caption{A simplified version of Figure~\protect\ref{fig:three-zigzags}. The upper
			red zigzag line as well as the lower blue zigzag line are, as before,
			translations of $\bigomega$. The $u$-coordinates of their convex and the
			concave corners coincide, up to a shift by~$\epsilon$, with their
			counterparts
			for $\bigomega$, see Figure~\ref{fig:transition}. The convention for convex and
			concave corners is the one from Figure~\ref{fig:transition}. The two black circles
			maintain their meaning from Figure~\protect\ref{fig:three-zigzags}.}
		
		\label{fig:three-zigzags-simplified}    
	\end{figure}
	
	\begin{lemma}
		The relationship between the Cauchy transforms  of $\referenceomega$ and
		$\bigomega$ is given by
		\begin{equation}
			\label{eq:cauchy-modified}
			\cauchy_{\referenceomega}(z) = \cauchy_\bigomega(z-\epsilon)\  P(z),
		\end{equation} 
		where
		\begin{equation}
			\label{eq:magic-product-Kerov} P(z) = 
			\frac{z-z_+}{z-z_-} 
			 \prod_{k=i}^{j-1}
				\frac{z-(\kerx_k+\epsilon)}{z- (\kery_{k+1}+\epsilon)}.
		\end{equation}
	\end{lemma}
	\begin{proof}
		An inspection of Figures~\ref{fig:three-zigzags} and \ref{fig:three-zigzags-simplified} shows
		that $\referenceomega$ has the
		following set of concave corners:
		\[ \kerx_0 + \epsilon, \dots, \kerx_{i-1} + \epsilon,\qquad 
		z_-,\qquad \kerx_{j}+\epsilon,  \dots, \kerx_{\kerell}+\epsilon, \]  
		and the following set of concave corners:
		\[ \kery_1 + \epsilon, \dots, \kery_{i} + \epsilon, \qquad 
		z_+,\qquad 
		\kery_{j+1}+\epsilon,  \dots, \kery_{\kerell}+\epsilon. \]  
		The proof is completed by comparing the zeros and the poles of the rational
		functions which appear on both sides of \eqref{eq:cauchy-modified}.
	\end{proof}
	Our strategy is to use this formula in order to express the cumulative function
	of $\referenceomega$ in terms of the cumulative function for $\bigomega$.

	\subsubsection{Identity fulfilled by convex and concave corners}
	
	\begin{lemma}
		With the above notations,
		\begin{equation}
			\label{eq:kerov-centered} z_+ - z_- = \epsilon+ \sum_{k=i}^{j-1} 
			\kery_{k+1} - \kerx_k.
		\end{equation}
	\end{lemma}
	\begin{proof}
		We will calculate the difference of the $x$-coordinates of the two black
		circles (the intersection points) in Figure~\ref{fig:three-zigzags} in two ways.
		
		On one hand, this difference is equal to the difference $z_+- z_-$ of their
		$u$-coordinates which is the left-hand side of \eqref{eq:kerov-centered}.
		
		On the other hand, the difference of the $x$-coordinates would
		decrease by~$\epsilon$ if we transform the intersection points as follows:
		\begin{itemize}
			\item we shift the left circle by the vector with the French coordinates
			$x=0$, $y=\epsilon$ (which is the opposite to the blue vector on
			Figure~\ref{fig:three-zigzags});
			
			\item we shift the right circle by the vector with the French coordinates
			$x=-\epsilon$, $y=0$ (which is the opposite to the red vector on
			Figure~\ref{fig:three-zigzags}).
		\end{itemize}
		The transformed  black circles both lie on the \conti diagram $\bigomega$.
		The difference of their $x$-coordinates is equal to the sum of the segments of
		$\bigomega$ which lie between the transformed circles and are parallel to the
		$x$-axis, see Figures~\ref{fig:three-zigzags} and \ref{fig:three-zigzags-simplified}, which is
		the second summand on the right-hand side of \eqref{eq:kerov-centered}.
	\end{proof}

	\subsubsection{The lower bound for the product $P$}

	\begin{lemma}
		\label{lem:lower-bound-for-P} 
		With the above notations, for each $z\geq z_+$ 
		\[ P(z) \geq P^{\min}_{z_+}(z),\]
		cf.~\eqref{eq:p-min}.
	\end{lemma}
	\begin{proof}
		In order to find a lower bound for $P=P(z)$,
		we shall view it as a function of the variables 
		\begin{equation} 
			\label{eq:my-little-variables} 
			\kerx_i \leq \kery_{i+1} \leq \kerx_{i+1} \leq
			\kery_{i+2} \leq \cdots  \leq \kery_{j-1} \leq \kerx_{j-1} \leq \kery_{j}  
			\in [z_- - \epsilon ,\; z_+  -\epsilon]
		\end{equation}
		which are subject to the constraint \eqref{eq:kerov-centered}. Let the
		variables \eqref{eq:my-little-variables} be such that the minimal value of $P$
		is attained (such a minimum is attained by the compactness argument). Let us
		modify two of these variables by replacing~$\kerx_m$ by $\kerx_m+\tau$, as well
		as $\kery_{m+1}$ by $\kery_{m+1}+\tau$, where $\tau<0$ is very close to zero. 
		Since 
		\[
		P(z)= \frac{z-z_+}{z-z_-} \prod_{k=i}^{j-1} \left[ 1+ 
		\frac{\kery_{k+1}-\kerx_k}{z- (\kery_{k+1}+\epsilon)} \right]
		\] 
		is an increasing function of~$\tau$, the minimality of $P$ implies
		that such a modification should not be possible. As a consequence, the
		appropriate inequalities in \eqref{eq:my-little-variables} must be saturated
		and
		\[ \kerx_i= z_- - \epsilon, 
		\qquad\qquad 
		\kery_{i+1} = \kerx_{i+1}, 
		\qquad\qquad 
		\kery_{i+2} = \kerx_{i+2}, 
		\qquad\qquad 
		\dots, 
		\qquad\qquad
		\kery_{j-1} = \kerx_{j-1}. \]
		Additionally, by \eqref{eq:kerov-centered}, 
		\[ \kery_j =  z_+ -2 \epsilon.\]
		The product \eqref{eq:magic-product-Kerov} evaluated for this particular
		choice of variables exhibits numerous cancellations, resulting in an elegant
		formula for $P(z)$ as shown below.
		
		In this way we proved that for $z\geq z_+$ the product
		\eqref{eq:magic-product-Kerov} is bounded from below by
		\[ P^{\min}_{z_+}(z) = 
		\frac{z-z_+}{z+\epsilon-z_+} =
		1 - \frac{\epsilon}{z+\epsilon-z_+},
		\]
		as required.
	\end{proof}

	\subsubsection{Proof of Proposition~\ref{prop:tails-zigzag}}
	\label{sec:lower-bound-tail}

	\begin{proof}[Proof of Proposition~\ref{prop:tails-zigzag}]
		The atoms of the transition measure
		$\mu_{\referenceomega}$ (respectively, the atoms of $\mu_{\bigomega}$) are
		given by the residues of the rational function $\cauchy_{\referenceomega}$
		(respectively, $\cauchy_{\bigomega}$). We use \eqref{eq:cauchy-modified} in
		order to relate the residues of $\cauchy_{\referenceomega}$ to the residues to
		$\cauchy_{\bigomega}$. 
		Since the residues of these two functions coincide up to a shift by $\epsilon$,
		it follows that
		\begin{multline}
			\label{eq:comparison}
			\int_{-\infty}^\infty   \phi(z)  \dif\mu_{\referenceomega}(z)  
			\geq 
			\int_{z_+ }^\infty
			\phi(z)  \dif\mu_{\referenceomega}(z) 
			= 
			\int_{z_+ - \epsilon}^\infty
			\phi(z+\epsilon) \ P(z+\epsilon) \dif \mu_{\bigomega}(z) \\ \geq   
			\int_{-\infty}^\infty  \phi(z+\epsilon) \  P^{\min}_{z_+}(z+\epsilon) \dif \mu_{\bigomega}(z), 
		\end{multline}
		where the inequality follows from Lemma~\ref{lem:lower-bound-for-P}, as required.
	\end{proof}
	
	\subsection{Lower bound for the transition measure of $\referenceomega$, generic diagrams}

	The content of the following result parallels Proposition~\ref{prop:tails-zigzag} in content; 
	however, a critical distinction emerges in that we no longer require $\bigomega$ to satisfy 
	the zigzag condition. Furthermore, rather than employing a test function $\phi$, our analysis 
	directly addresses the distribution tail.
	\begin{proposition}
		\label{prop:tails} 
		Let $\epsilon>0$, $\bigomega$, $f$ and $z_+$ be as described
		in Sections~\ref{sec:input} and \ref{sec:intersection}, and let $\referenceomega=
		\referenceomega_{\Omega,f,\epsilon}$ be the corresponding shifted diagram. Then
		\[        \mu_{\referenceomega}\left( (z_+, \infty) \right)  \geq	 
		\int_{-\infty}^\infty  P_{z_+}^{\min}(z+\epsilon) \dif \mu_{\bigomega}(z), \]
		cf.~\eqref
		{eq:p-min}.
	\end{proposition}
	The remainder of this section presents our proof: we approximate $\bigomega$ 
	with zig-zag diagrams, apply Proposition~\ref{prop:tails-zigzag},
	and employ continuity arguments.

	\subsubsection{Approximation by zigzags}
	
	Let $(\Omega^{(n)})$ be a sequence of zigzag diagrams which converges pointwise to $\Omega$, and such that the \conti
	diagrams $(\Omega^{(n)})$ have a common compact support, i.e., there exists a
	constant $C$ such that 
	\begin{equation}
		\label{eq:common-compact}
		\Omega^{(n)}(z)=|z| \qquad \text{if $|z|>C$} 
	\end{equation} 
	holds for any integer $n\geq 1$. Such a sequence can be explicitly constructed
	similarly as in the proof of the general case of
	Proposition~\ref{prop:comparison-of-transition}. Since the transition measure provides a
	homeomorphism between appropriate topological spaces, it follows that the
	sequence of transition measures $\big( \mu_{\Omega^{(n)}} \big)$ converges
	weakly to the transition measure $\mu_{\Omega}$.

	\subsubsection{Asymptotics of $z_+^{(n)}$}
	
	The constructions from Sections~\ref{sec:intersection} and \ref{sec:reference-diagram} can be
	applied to $\Omega:= \Omega^{(n)}$. We denote by 
	$\referenceomega^{(n)}=\referenceomega_{\bigomega^{(n)},\epsilon,f}$ 
	and $z_+^{(n)}$ and
	the
	corresponding values of $\referenceomega$ and $z_+$.  
	Let 
	\begin{equation}
		\label{eq:limsup}
		z_+^{(\infty)} := \limsup_{n\to\infty} z_+^{(n)}.
	\end{equation}
	
	\begin{lemma}
		\label{lem:zinfty}
		With the above notations,
		\begin{equation}
			\label{eq:zinfty}
			z_- + \epsilon \leq	 z_+^{(\infty)} \leq  z_+. 
		\end{equation}
	\end{lemma}
	\begin{proof}
		In the context of  $\Omega:= \Omega^{(n)}$, 
		the defining property of $z_+^{(n)}$, i.e., equation 
		\eqref{eq:zetplus}, takes the form
		\[	f\big(z_+^{(n)}\big) = \bigomega^{(n)}\big(z_+^{(n)} -\epsilon\big) + \epsilon.  \]
		The $1$-Lipschitz continuity of all relevant functions implies that 
		\[	f\big(z_+^{(\infty)}\big) = \bigomega\big(z_+^{(\infty)} -\epsilon\big) + \epsilon.  \]
		In other words, $ z_+^{(\infty)}$ is a solution of \eqref{eq:zetplus}; by the
		maximality of $z_+$, the second inequality in \eqref{eq:zinfty} follows.
		
		\medskip

		On the other hand, any solution of the equation \eqref{eq:zetplus} is bigger by
		at least~$\epsilon$ than any solution of \eqref{eq:zminus}. The first inequality
		in \eqref{eq:zinfty} follows.
	\end{proof}

	\subsubsection{Asymptotics of\/ $\referenceomega^{(n)}$}

	From the alternative definition \eqref{eq:minmax} of the shifted diagram it follows that the sequence
	of shifted diagrams $\big( \referenceomega^{(n)} \big)$ has a common compact
	support and converges pointwise to the shifted diagram $\referenceomega$. As a
	consequence, the sequence of their transition measures $\big(\mu_{
		\referenceomega^{(n)} } \big)$ converges to $\mu_{\referenceomega}$ in the weak
	topology of probability measures.

	\subsubsection{Applying Proposition~\ref{prop:tails-zigzag} to zigzag diagrams}
	
	Let $\phi_k \colon \R \to \R_+$ be a continuous function which will be specified later. 
	Proposition~\ref{prop:tails-zigzag} implies that
	\begin{equation}
		\label{eq:zigzagsok}    
		\int_{-\infty}^\infty   \phi_k(z)  \dif\mu_{\referenceomega^{(n)}}(z)  \geq	 
		\int_{-\infty}^\infty  \phi_k(z+\epsilon)\ P_{z^{(n)}_+}^{\min}(z+\epsilon) 
		\dif \mu_{\bigomega^{(n)}}(z)
	\end{equation}
	holds true for each $n$. 

	\medskip
	
	Lemma~\ref{lem:zinfty} implies that for each $\delta>0$ 
	\[ 
	z_+^{(n)} \leq \delta+ z_+\]
	holds true for all $n\geq n_0$ for some value of $n_0$. 
	Since for each fixed value of
	$z$, the map $t \mapsto P_{t}^{\min}(z)$ is weakly decreasing, this implies that
	the right hand side of \eqref{eq:zigzagsok} can be bounded from below as follows 
	\begin{equation}
		\label{eq:rhs}
		\int_{-\infty}^\infty  \phi_k(z+\epsilon)\  P_{z_+^{(n)}}^{\min}(z+\epsilon) \dif \mu_{\bigomega^{(n)}}(z) \geq 
		\int_{-\infty}^\infty  \phi_k(z+\epsilon)\  P_{\delta+ z_+}^{\min}(z+\epsilon) \dif \mu_{\bigomega^{(n)}}(z)
	\end{equation}
	for $n\geq n_0$.

	We chain the inequalities \eqref{eq:zigzagsok} and \eqref{eq:rhs} and take
	iterated limit: first $n\to\infty$ and then $\delta\to 0$ and get:
	\begin{equation}
		\label{eq:phiphiphi}
		\int_{-\infty}^\infty   \phi_k(z)  \dif\mu_{\referenceomega}(z) \geq 
		\int_{-\infty}^\infty  \phi_k(z + \epsilon)\  P_{ z_+}^{\min}(z+\epsilon) \dif \mu_{\bigomega}(z).
	\end{equation}
	The passage to the limit as $n\to\infty$ is justified by the weak convergence of the transition measures
	of $\mu_{\referenceomega^{(n)}}$ to the transition measure of $\referenceomega$, coupled with
	the analogous weak convergence of the transition measures $\bigomega^{(n)}$ to the transition measure
	of $\bigomega$.

	\subsubsection{Approximating the step function}
	For an integer $k\geq 1$ we will specify now the function $\phi_k$ as follows:
	\[ \phi_k(z_+ + \Delta) = \begin{cases}
		0   & \text{for $\Delta\leq 0$}, \\
		k\Delta  & \text{for $0\leq \Delta \leq \frac{1}{k}$}, \\
		1   & \text{for $\Delta> \frac{1}{k}$}.
	\end{cases}
	\]
	In this way $(\phi_k)$ is a sequence of continuous functions which converges pointwise to 
	\[ \phi(z) = \begin{cases}
		0   & \text{for $z\leq z_+$}, \\
		1   & \text{for $z> z_+$}.
	\end{cases}
	\]
	We consider \eqref{eq:phiphiphi} for this specific choice of $\phi_k$ and take
	the limit $k\to\infty$:
	\[
	\mu_{\referenceomega}\big( (z_+, \infty) \big) \geq 
	\int_{-\infty}^\infty P_{ z_+}^{\min}(z+\epsilon) \dif \mu_{\bigomega}(z),
	\]
	which completes the proof of Proposition~\ref{prop:tails}.

	\section{The bounds for the cumulative function}
	\label{sec:bounds}
	
	\subsection{The upper bound}
	\begin{theorem}
		\label{theo:module-continuity-A}
		Let $\bigomega\colon \R \to \R_+$ be  a \conti diagram and let $\epsilon>0$
		and $z_0\in\R$  be fixed. Assume that $\zzeroplus$ is the maximal solution
		of the equation
		\begin{equation} 
			\label{eq:equation-for-Kerov}
			\bigomega(\zzeroplus-\epsilon) - \bigomega(z_0-\epsilon)= \zzeroplus - z_0 - 2 \epsilon
		\end{equation}
		see Figure~\ref{fig:what-is-zzero} for an illustration.
		
		\smallskip
		
		Then for any \conti diagram $\smallomega$ such that 
		$\dxy(\bigomega, \smallomega) \leq
		\epsilon$ the following upper bound for the cumulative
		function of $\smallomega$ holds true:
		\begin{multline}
			\label{eq:bound} 
			\lim_{\tau \to 0^+}  K_{\smallomega}(z_0-\tau)  \leq  1- \int_{\zzeroplus-\epsilon}^{\infty}   
			\left[ 1- \frac{\epsilon}{z+2\epsilon- \zzeroplus}\right]  \dif \mu_{\bigomega}(z)
			 \\ =
			\int_{-\infty}^{\infty}
			\left( \mathbbm{1}_{(-\infty,\ \zzeroplus-\epsilon)}(z) + \mathbbm{1}_{[\zzeroplus-\epsilon,\ \infty)}(z)\  \frac{\epsilon}{z+2\epsilon- \zzeroplus} \right) \dif \mu_{\bigomega}(z).
		\end{multline}
	\end{theorem}
	
	Note that equation \eqref{eq:equation-for-Kerov} may not have any solutions, or
	the set of solutions may be unbounded from above; the above result does not
	concern such cases. Since $\Omega$ is $1$-Lipschitz, it follows that each
	solution of the equation \eqref{eq:equation-for-Kerov} fulfills 
	\[ \zzeroplus \geq z_0+ \epsilon,\]
	see Figure~\ref{fig:what-is-zzero}. 
	
	The integrand on the right-hand side of
	\eqref{eq:bound} can be heuristically interpreted as a continuous, smoothed-out
	version of the indicator function $\mathbbm{1}_{(-\infty, z_0]}$.
	
	Some readers may find it aesthetically more pleasing to replace in
	\eqref{eq:equation-for-Kerov} and \eqref{eq:bound} the quantity $\zzeroplus$ by
	$\hat{z}:= \zzeroplus-\epsilon$.
	
	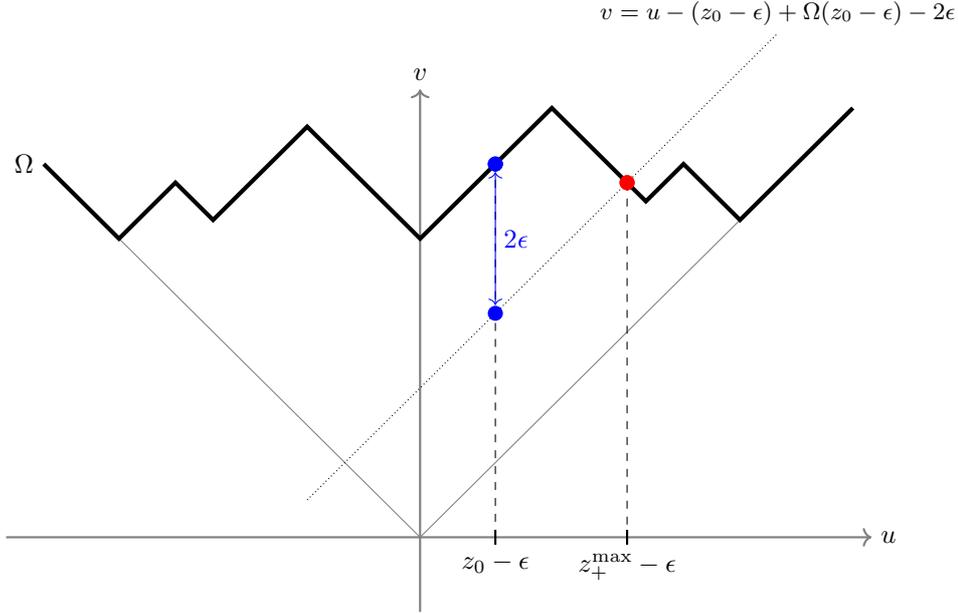
\begin{figure}
		\begin{tikzpicture}[scale=0.35, rotate=45]
			
			\node[fill,blue, circle, inner sep=2pt] (a1) at (12,8) {};
			\node[fill,red, circle, inner sep=2pt] (c1) at (15,4) {};

			\draw[->, thick, black!50]  (-11,11) -- (12,-12) node[anchor=
			west,black]{$u$};
			\draw[->, thick, black!50]  (-2,-2) -- (12,12) node[anchor=
			south,black]{$v$};
			\draw[thin, black!50] (0,20) -- (0,0) -- (20,0);

			\draw[dashed] (12,8) -- (2,-2) ;
			\draw[thick] (2,-2) +(0.2,0.2) -- +(-0.2,-0.2) node[anchor=north]
			{$z_0-\epsilon$};
			
			\draw[dashed] (15,4) -- (5.5,-5.5) ;
			\draw[thick] (5.5,-5.5) +(0.2,0.2) -- +(-0.2,-0.2) node[anchor=north]
			{$\zzeroplus- \epsilon$};

			\draw[densely dotted] (8,4) +(-10,0) -- +(15,0) 
			node[anchor=south] {\small$v= u-(z_0-\epsilon) + \bigomega(z_0-\epsilon)-
				2\epsilon $};

			\draw[ultra thick] (0,20) node[anchor=east]
			{\textcolor{black}{$\bigomega$}}
			-- ++(0,-4) 
			-- ++(3,0) -- ++(0, -2) 
			-- ++(5,0) -- ++(0, -6) 
			-- ++(7,0) -- ++(0, -5) 
			-- ++(2,0) -- ++(0, -3) 
			-- ++(6,0)  ;
			; 
			
			\node[fill,blue,opacity=0.2, circle, inner sep=2pt] (a) at (12,8) {};
			\node[fill,blue, circle, inner sep=2pt] (b) at (8,4) {};
			\node[fill,red,opacity=0.2, circle, inner sep=2pt] (c) at (15,4) {};
			\draw[<->,blue] (a) -- (b) node[midway, right] {$2 \epsilon$};

		\end{tikzpicture}
		\caption{The solid black zigzag line is the \conti diagram $\bigomega$.
			The densely dotted diagonal line is the plot of the affine function $z \mapsto 
			z -(z_0-\epsilon) + \bigomega(z_0-\epsilon)- 2\epsilon $. The rightmost, red
			circle indicates the intersection of this line with the plot of $\bigomega$.}
		\label{fig:what-is-zzero}
		
	\end{figure}
	
	\begin{proof}
		Our strategy will be to apply Proposition~\ref{prop:steepest}. We revisit the setup considered
		in Sections~\ref{sec:input}--\ref{sec:intersection}. Let the affine function $f$ be
		given by \eqref{eq:f-for-z0} and, as before, let $z_+$ be the maximal solution
		of \eqref{eq:zetplus}.
		In order to
		find an upper bound for $z_+$ we notice that such a maximal value corresponds
		to the affine function $f$ with the smallest possible $v$-intercept, see Figure~\ref{fig:three-zigzags}.
		
		The inequality \eqref{eq:blue-bound} taken at $z=z_0$ implies that the affine
		function~$f$ is bounded from below by
		\[  f_{\min}(z) = z - z_0 + \bigomega(z_0-\epsilon) - \epsilon.\]
		With these notations, \eqref{eq:equation-for-Kerov} can be rewritten in the form
		\begin{equation}
			\label{eq:rewritten}
			f_{\min}( \zzeroplus) = \bigomega(\zzeroplus -\epsilon) + \epsilon.   
		\end{equation}
		By comparing it with the defining property \eqref{eq:zetplus}
		it follows that $\zzeroplus$ is a special case of~$z_+$ evaluated for the minimal affine function $f_{\min}$.	
		In this way we proved that $z_+ \leq \zzeroplus$.

		Furthermore, the degenerate case $\omega(z_0) = z_0$ is not possible because it
		would imply that $f_{\min}(z) \leq f(z) = z$ lies below the $x$-axis. As a
		consequence, the set of solutions of \eqref{eq:rewritten} (which is equivalent
		to \eqref{eq:equation-for-Kerov}) would be empty or not bounded from above which
		contradicts the assumption on $\zzeroplus$. It follows that without the loss of generality we
		may assume $\omega(z_0) > z_0$ and Proposition~\ref{prop:steepest} is applicable.
		
		\medskip	
		
		We apply Proposition~\ref{prop:steepest}; we denote by $\referenceomega$ the shifted
		diagram which is defined there. The left limit of the cumulative function of $\omega$ is bounded by
		\begin{equation}
			\label{eq:c1}
			\lim_{\tau \to 0^+} K_{\omega}(z_0 - \tau) \leq
			K_{\referenceomega}(z_0)= 
			1- \mu_{\referenceomega}\big( (z_0,\infty)  \big).
		\end{equation}

		We will use the notations from Figure~\ref{fig:three-zigzags}. Let $P$ be the point on
		the plane which has the Russian coordinates $u= z_0$ and $v=\omega(z_0)$. From
		the way the affine function $f$ was defined in Proposition~\ref{prop:steepest} it follows
		that $P$ belongs both to the plot of the function $\omega$ as well as to the
		plot of the affine function~$f$. We can rephrase it as follows: $P$ belongs to
		the intersection of the area between the red and the blue curve (cf.~Section~\ref{sec:red-and-blue}) and the plot of $f$. It follows immediately that 
		\[z_0 \leq z_+,\]
		hence
		\begin{equation}
			\label{eq:c2}
			\mu_{\referenceomega}\big( (z_0,\infty)  \big)  \geq
			\mu_{\referenceomega}\big( (z_+,\infty)  \big).
		\end{equation}
		
		The right-hand side can be bounded from below by Proposition~\ref{prop:tails}.
		Since for each fixed value of
		$z$, the map $t \mapsto P_{t}^{\min}(z)$ is weakly decreasing, it follows that 
		\begin{equation}
			\label{eq:c3}
			\mu_{\referenceomega}\big( (z_+,\infty)  \big) \geq 
			\int_{-\infty}^\infty  P_{(\zzeroplus)}^{\min}(z+\epsilon) \dif \mu_{\bigomega}(z).
		\end{equation}
		
		We complete the proof by combining the inequalities provided by \eqref{eq:c1}, \eqref{eq:c2} and \eqref{eq:c3}.
	\end{proof}

	\subsection{Transpose of a diagram}
	
	If $\omega\colon \R\to \R_+$ is a \conti diagram we define its \emph{transpose}
	\[ \omega^T(z) = \omega(-z).\]
	This definition extends the usual notion of a transposed (or \emph{conjugate}) Young diagram. 
	
	It is easy to check that the transition measure of a transposed diagram
	$\mu_{(\omega^T)}$ is the push-forward of the transition measure of the original
	diagram $\mu_{\omega}$ under the involution $\R \ni z \mapsto -z \in \R$.

	\subsection{The lower bound}
	
	The following is a mirror image of Theorem~\ref{theo:module-continuity-A}.
	
	\begin{theorem}
		\label{theo:module-continuity-B} 
		
		Let $\bigomega\colon \R \to \R_+$ be  a \conti diagram and let $\epsilon>0$
		and $z_0\in\R$  be fixed. Assume that $\zzerominus$ is the minimal solution of
		the equation
		\[ 		\bigomega(\zzerominus + \epsilon) - \bigomega(z_0 + \epsilon)= z_0 - \zzerominus  - 2 \epsilon
		. \]

		Then for any \conti diagram $\smallomega$ such that 
		$\dxy(\bigomega, \smallomega) \leq 
		\epsilon$ the following lower bound for the cumulative 
		function of $\smallomega$ holds true:
		\[
			K_{\smallomega}(z_0)  \geq  
			\int_{-\infty}^{\zzerominus+ \epsilon} \left[ 1- \frac{\epsilon}{\zzerominus + 2 \epsilon- z } \right] \dif
			\mu_{{\bigomega}}(z).\]
	\end{theorem}
	\begin{proof}
		Consider the involution of the real line  $z \mapsto -z$
		and apply 
		Theorem~\ref{theo:module-continuity-A} for transposed diagrams 
		\[\omega':= \omega^T, \qquad\qquad \Omega':= \Omega^T,\] 
		and for 
		\[ z_0' := - z_0, \qquad\qquad \zzeroplus{}':= -\zzerominus . \qedhere \]	
	\end{proof}

	\section{Proof of the main result, Theorem~\ref{theo:main-modulus}}
	\label{sec:proof-of-main-result}
	
	\begin{lemma}
		\label{lem:zzerozero}
		We keep the assumptions from Theorem~\ref{theo:main-modulus}. 		
		There exists a constant $\epsilon_0>0$ with the property that for each $z\in
		[a_0,b_0]$ and for each $\epsilon \in (0,\epsilon_0)$
		the quantity $\zzeroplus$ considered in Theorem~\ref{theo:module-continuity-A}
		 is well-defined, and fulfills the inequalities
		\[  \zzeroplus  - z_0 \leq \frac{2\epsilon}{\delta} \qquad \text{and} \qquad \zzeroplus \leq \frac{b_0+b}{2}. \]
	\end{lemma}
	\begin{proof}
		Let us fix $z_0\in[a_0,b_0]$ and $\epsilon>0$. We consider a continuous, weakly decreasing
		function
		\[	F(z) = \bigomega(z-\epsilon)- z - \bigomega(z_0-\epsilon) + z_0 + 2\epsilon;
		\]
		the motivation for this definition is that the condition \eqref{eq:equation-for-Kerov} is
		equivalent to $F(\zzeroplus)=0$. 
		We start with the observation that $F(z_0)= 2 \epsilon >0$ is strictly positive.
		
		Assume that $\epsilon \in (0,\ a_0 - a)$ is sufficiently small. 
		The assumption \ref{enum:Lipschitz} of Theorem~\ref{theo:main-modulus} 
		implies that if $z\in[z_0,  b]$ then 
		\begin{equation}
			\label{eq:lipschitz-trick} F(z) \leq F(z_0)  -(z-z_0)\ \delta = 2\epsilon 
			- (z-z_0)\  \delta.
		\end{equation}
		We denote by 
		\[ z_{\max}:=z_0 + \frac{2 \epsilon}{\delta}  \]
		the value of the variable $z$ for which the right-hand side of
		\eqref{eq:lipschitz-trick} is equal to zero. Note that for $z>z_{\max}$ this right hand side is strictly negative.
		We also define 
		\[ \epsilon_0 = \min\left( \frac{(b-b_0)\ \delta }{4}, \  a_0-a  \right).\]
		In this way, if $0< \epsilon < \epsilon_0$ then 
		\[ z_{\max} < z_0+ \frac{b- b_0}{2} \leq \frac{b_0+b}{2}<b.\]  
		As a consequence, \eqref{eq:lipschitz-trick}  is applicable for $z:= z_{\max}$
		and $F(z_{\max}) \leq 0$. Furthermore, $F$ is strictly negative on the interval
		$(z_{\max},\infty)$.
		As a consequence, the solutions of the equation $F(z)=0$ form a non-empty, bounded set, so $\zzeroplus$ is well-defined, and it fulfills the bound
		$\zzeroplus\leq z_{\max}$ which completes the proof.
	\end{proof}
	
	\begin{proof}[Proof of Theorem~\ref{theo:main-modulus}]
		Let $\epsilon_0$ be the value provided by 
		Lemma~\ref{lem:zzerozero}. Without loss of generality we may assume that $0< \epsilon_0 <\frac{1}{2}$.
		
		\smallskip
		
		We start with the case  $\epsilon\in (0, \epsilon_0)$.
		Theorem~\ref{theo:module-continuity-A} implies that for any $z_0\in [a_0,b_0]$
		\begin{multline}
			\label{eq:darowanej-nierownosci-nie-patrzy-sie-w-zeby}
			\lim_{\tau \to 0^+} K_{\smallomega}(z_0-\tau) -	K_{\bigomega}(z_0) \\ \leq  
			\mu_{\bigomega}\big( (z_0, \zzeroplus - \epsilon] \big) +
			\epsilon \int_{\zzeroplus - \epsilon}^b \frac{1}{z  + 2\epsilon - \zzeroplus} \dif
			\mu_{{\bigomega}}(z) 
			+ 
			\epsilon \int_{b}^\infty \frac{1}{z  + 2\epsilon - \zzeroplus} \dif
			\mu_{{\bigomega}}(z).		
		\end{multline}

		We denote by $\rho_{\max}$ any upper bound for the density of the measure
$\mu_\bigomega$ on the interval $[a,b]$. 
		By the first inequality provided by Lemma~\ref{lem:zzerozero}, 
		the first summand on the right hand side of
		\eqref{eq:darowanej-nierownosci-nie-patrzy-sie-w-zeby} is bounded by
		\begin{equation}
			\label{eq:summand-1}
					 \frac{2 \epsilon \rho_{\max}}{\delta}  .
		\end{equation}
		The second summand is bounded from above by
		\begin{equation}
			\label{eq:summand-2} 
			\epsilon \rho_{\max} \log \frac{b+\epsilon_0 -a_0}{\epsilon}.
		\end{equation}
		The third summand is bounded from above by
		\begin{equation}
			\label{eq:summand-3} \frac{2 \epsilon}{b - b_0}  
		\end{equation}
		because of the second inequality provided by Lemma~\ref{lem:zzerozero}.		
		We can choose a sufficiently large value of $C_1$ such that the sum of the
		contributions of \eqref{eq:summand-1}, \eqref{eq:summand-2}, and
		\eqref{eq:summand-3} is bounded from above by $C_1 \epsilon \log \frac{1}{\epsilon}$
		for each $\epsilon\in (0, \epsilon_0)$.
		
		In this way we proved that 
		\[ 	\lim_{\tau \to 0^+} K_{\smallomega}(z_0-\tau) -	K_{\bigomega}(z_0) \leq 
		C_1 \epsilon \log \frac{1}{\epsilon}. \]
		Since this bound is uniform over $z_0$ it follows that on the left-hand side we
		can replace the left limit of the cumulative function by the cumulative function
		$K_{\omega}(z_0)$ itself.

		\smallskip
		
		For the remaining case $\epsilon \in \left[ \epsilon_0, \frac{1}{2} \right)$
		we can choose $C_2$ to be sufficiently large that
		\[  C_2 \epsilon \log \frac{1}{\epsilon} > 1, 
		      \qquad \text{for each }\epsilon \in \left[ \epsilon_0, \frac{1}{2} \right) \]
		so that the inequality \eqref{eq:main-theorem} is trivially fulfilled for any $C\geq C_2$.
				
		\smallskip
				
		In this way we proved that the upper bound
		\[ K_{\smallomega}(z_0) - K_{\bigomega}(z_0) \leq C \epsilon \log \frac{1}{\epsilon} \]
		holds true for any sufficiently large value of $C \geq \max(C_1, C_2)$.
		
		\bigskip
		
		The lower bound proof follows by analogous reasoning to the upper bound case, so
		we omit it for brevity. To reconstruct this argument, one need only replace each
		reference to Theorem~\ref{theo:module-continuity-A} with 
		Theorem~\ref{theo:module-continuity-B} throughout the preceding
		derivation. Additionally, note that an analogue of Lemma~\ref{lem:zzerozero} is required,
		addressing the quantity $\zzerominus$ that appears in
		Theorem~\ref{theo:module-continuity-B}.
	\end{proof}

\def\cprime{$'$}
\providecommand{\bysame}{\leavevmode\hbox to3em{\hrulefill}\thinspace}
\providecommand{\MR}{\relax\ifhmode\unskip\space\fi MR }
\providecommand{\MRhref}[2]{%
	\href{http://www.ams.org/mathscinet-getitem?mr=#1}{#2}
}
\providecommand{\href}[2]{#2}

\end{document}